%% file: priestley_perspective_in_pointfree_topology.tex
\documentclass[reqno,oneside]{amsart}

\input{preamble}

\title{Priestley perspective on pointfree topology}
\author{Guram Bezhanishvili}
\address{New Mexico State University, USA}
\email{guram@nmsu.edu}

\author{Sebastian D. Melzer}
\address{University of Salerno, Italy}
\email{smelzer@unisa.it}

\subjclass[2020]{18F70; 06D22; 06D50; 06E15; 54D10; 54D15; 54D30; 54D45}
\keywords{Priestley duality; pointfree topology; spatiality; sublocale; separation axioms; compactness; local compactness; algebraic frame; coherent frame; Stone frame}

\begin{document}

\begin{abstract}
    Priestley duality has diverse applications in various branches of mathematics. 
    In this survey, we discuss its usefulness
    in pointfree topology. This is done by providing Priestley perspective on several key notions, including spatiality, sublocales, separation axioms, compactness, and local compactness.
    This approach yields a new perspective on
    a number of classic results in pointfree topology. 
\end{abstract}

\maketitle

\tableofcontents

\hfill \emph{Tribute to Hilary Priestley}   

\section{Introduction}

It is a celebrated result of Priestley \cite{Pri70} that the category $\DLat$ of bounded distributive lattices is dually equivalent to the category $\Pries$ of Priestley spaces. The latter are Stone spaces equipped with a partial order that satisfies Priestley separation (see \cref{sec: PD for frames} for details). As such, Priestley duality generalizes Stone duality for boolean algebras \cite{Sto36}. Priestley spaces form an important subcategory of the category of compact ordered spaces studied by Nachbin \cite{Nac65}. The correspondence is roughly the same as between compact Hausdorff spaces and Stone spaces, and many results about Stone spaces have their natural generalization to Priestley spaces (see, e.g., \cite{Pri84,BB+10}). 

Each Priestley space carries two natural spectral topologies of open upsets and open downsets, and the Stone topology of the Priestley space is the join of these two
(see, e.g., \cite[Prop.~3.4]{BB+10}). 
This correspondence is at the heart of the Cornish isomorphism \cite{Cor75} between $\Pries$ and the category $\Spec$ of spectral spaces. The latter are 
the Zariski spectra of commutative rings (see, e.g., \cite{AM69} or \cite{DST19}). 
In fact, the first duality for $\DLat$ was developed by Stone \cite{Sto38} using the language of spectral spaces. However, it wasn't until Priestley's work that this has become a major tool in lattice theory and related areas.  
Indeed, since Priestley spaces are compact Hausdorff, the Priestley approach is often easier and the use of order only helps the intuition by allowing to draw the Hasse diagram of the situation at hand. 
Because of this, Priestley duality found numerous applications, not only in lattice theory, but also 
in several other areas of mathematics. For example, Priestley duality is used in algebra (see, e.g., \cite{Pri94,Mar96,CD98,CC06,DG+11}), topology (see, e.g., \cite{Pri72,Han86,KS91,Nai00,BPS09,BM11b}),  
logic (see, e.g., \cite
{DH01}), domain theory (see, e.g., 
\cite{GH+03,YL06}), theoretical computer science (see, e.g., \cite{GGP08,Geh16,BHM23}), etc. 

In this article, we make a case that pointfree topology provides another fertile ground for the use of Priestley duality. 
As the name indicates, pointfree topology is the study of topological spaces through their lattices of open sets, without an explicit reference to their points. 
This perspective originated in the 1930s with the work of Stone
\cite{Sto36,Sto38}, Tarski \cite{Tar38} 
and Wallman
\cite{Wal38}. It was further developed in the 1940s 
\cite{MT44, MT46, Noeb48} (see also \cite{Noeb54,RS63}),
and the 1950s \cite{Ehr57,Ben59,PP59}. 
Ehresmann and Bénabou started to view
complete Heyting algebras as generalized topological spaces, which they called \emph{local lattices}. The two terms that are in use today are \emph{locale} and \emph{frame}. The first was coined by Isbell \cite{Isb72} and the second by
Dowker and Papert \cite{DP66}, with the two categories of locales and frames being opposite 
of each other. 
The dual adjunction between the categories of topological spaces and frames was established in \cite{DP66}. 
Isbell’s influential paper \cite{Isb72} 
is considered the birth of pointfree topology as an independent branch of topology (and lattice theory).
It was further developed by Banaschewski, Johnstone, and Simmons (among others), and matured rapidly. For a comprehensive account of the subject and its history, we refer to Johnstone’s monograph \cite{Joh82} and more recent books by Picado and Pultr \cite{PP12,PP21}.

While pointfree topology provides an elegant algebraic framework 
to study topological notions, it can sometimes obscure the geometric intuition of a topologist.
Priestley duality offers a way to recover this intuition through the Priestley dual of a frame.
This idea appeared first in a note by Wigner \cite{Wig79}, who worked with the spectral space of a frame, but the patch topology played a central role in his considerations.
Pultr and Sichler \cite{PS88} 
identified the subcategory of Priestley spaces that is dually equivalent to the category of frames and characterized compact and regular frames. 
This perspective has been extended and refined in the 2000s by various authors
\cite{PS00,BG07,BGJ13,BGJ16,AB+20,AB+21}, and an alternative spectral approach was considered in \cite{Sch13,Sch17a,Sch17b} (see also \cite{DST19}).
In our recent work, we have developed a systematic study of Priestley spaces of frames (see \cite{BM22,BM23,BM25,BBM25,Mel25}). 
On the one hand, it provides a new perspective on some classic results in pointfree topology; on the other hand, it can be used to address some open problems in the field. 
It should be noted, however, that Priestley duality relies on the Prime Ideal Theorem (PIT), and some results even require the full strength of the Axiom of Choice (AC).

In this survey, 
we illustrate how Priestley duality can be used to gain a grasp of various categories of frames that play a prominent role in pointfree topology. This includes the Priestley characterization of spatiality, sublocales, 
various separation axioms, 
compactness, and local compactness, as well as the description of Priestley spaces of algebraic, coherent, and Stone frames. 
While most of the results we present are known, the uniform approach to separation axioms 
is new, as are some of the proofs, which help the reader gain 
an understanding of the type of reasoning
involved in such considerations. 

The paper is structured as follows. In \cref{sec: PD for frames}, we recall Priestley duality for distributive lattices. 
We then characterize exact joins, which paves the way 
to Priestley duality for frames. In \cref{sec: spatiality}, we restrict Priestley duality to spatial frames. In \cref{sec: nuclei}, we describe sublocales of a frame in terms of special closed subsets of the Priestley dual of the frame, yielding an alternative proof of Isbell's Density Theorem. In \cref{sec: sep axioms}, we develop the theory of kernels in Priestley spaces of frames, and show that various separation axioms can be described uniformly utilizing the language of kernels.
In \cref{sec: compact}, we characterize Priestley spaces of compact frames, yielding a new proof of Isbell's Spatiality Theorem, which gives rise to a new proof of Isbell duality for compact regular frames. The above characterization is then generalized to continuous frames, yielding a new proof 
of Hofmann--Lawson duality. Finally, in \cref{sec: algebraic}, we describe the Priestley spaces of algebraic, coherent, and Stone frames, thus giving rise to new proofs of several well-known dualities in the literature. 
We conclude by 
deriving Priestley and Stone dualities from this approach.

\section{Priestley duality for frames} \label{sec: PD for frames}

Priestley duality is a convenient tool to reason about distributive lattices using the language of ordered topological spaces. We briefly recall basic definitions. 
Let $X$ be a poset. 
As usual, for $S \subseteq X$,  we write $$\upset S := \{x \in X \mid s \leq x \text{ for some } s \in S\}$$ and $$\downset S := \{x \in X \mid s \geq x \text{ for some } s \in S\}.$$ Then $S$ is an \emph{upset} if 
$S = \upset S$ and a \emph{downset} if 
$S = \downset S$. 
A subset of a topological space $X$ is \emph{clopen} if it is both closed and open, and $X$ is {\em zero-dimensional} provided clopen subsets form a basis.

\begin{definition}[Priestley spaces]
\leavevmode
\begin{enumerate}
    \item A \emph{Priestley space} is a compact topological space $X$ equipped with a partial order $\leq$ satisfying \emph{Priestley separation}: 
    \[ 
    x \nleq y \implies \exists \text{ clopen upset } U \text { such that } x \in U \text{ and }y \notin U.
    \] 
    \item A \emph{Priestley morphism} is a continuous order-preserving map between Priestley spaces.
    \item Let $\Pries$ be the category of Priestley spaces and Priestley morphisms.
\end{enumerate}
\end{definition}

Letting $\DLat$ denote the category of bounded distributive lattices and bounded lattice morphisms, we have (see \cite{Pri70,Pri72}):

\begin{theorem}[Priestley duality]
    $\DLat$ is dually equivalent to $\Pries$.
\end{theorem}

\begin{remark} 
\leavevmode
\begin{enumerate}[cref=remark]
    \item \label{rem: Priestley functors}
The functors establishing Priestley duality are constructed as follows.  
The functor from $\DLat$ to $\Pries$ sends a bounded distributive lattice $D$ to the Priestley space $X_D$ of prime filters of $D$ ordered by inclusion and topologized by the basis
\[
    \{ \sigma(a) \cap \sigma(b)^c \mid a,b \in D \},
\]
where $\sigma(a) = \{x \in X_D \mid a \in x\}$.  
A bounded lattice morphism $h \colon D \to E$ is mapped to the Priestley morphism $h^{-1} \colon X_E \to X_D$.  

The functor from $\Pries$ to $\DLat$ sends a Priestley space $X$ to the lattice $\clopup(X)$ of its clopen upsets, and a Priestley morphism $f \colon X \to Y$ is mapped to the bounded lattice morphism $f^{-1} \colon \clopup(Y) \to \clopup(X)$.
 
\item We will often use 
that  $\sigma : D \to \clopup(X_D)$ is a lattice isomorphism.
The injectivity of $\sigma$ requires PIT,
and so does the compactness of $X_D$ as it 
is established using Tychonoff's Theorem for Hausdorff spaces \cite[p.~187]{Pri70}. 
In fact, 
Priestley duality is equivalent to PIT (see, e.g., \cite[p.~79]{Joh82}).\label{rem:PIT}
\end{enumerate}
\end{remark}

We recall some basic properties of Priestley spaces that will be used throughout the paper (see, e.g., \cite{Pri84} or \cite{BB+10}). 
For a poset $X$, we write $\min X$ and $\max X$ for the sets
of minimal and maximal points, respectively.

\begin{lemma}[Basic facts about Priestley spaces]\label{lem: pri-facts}
Let $X$ be a Priestley space.
\begin{enumerate}[cref=lemma]
    \item $X$ is a Stone space \<compact, Hausdorff, zero-dimensional\>.\label{lem: pri-facts-stone}
    \item The upset and  downset of a closed set are closed.\label{lem: pri-facts-upclosed}
    \item Each closed upset \<resp.\ downset\> is an intersection of clopen upsets \<resp.\ downsets\>.\label{lem: pri-facts-intersection}
    \item Each open upset \<resp.~downset\> is a union of clopen upsets \<resp.~down\-sets\>. \label{lem: pri-facts-unions}
    \item For each closed $F \subseteq X$ and $x \in F$, both $\downset x \cap \min F$ and $\upset x \cap \max F$ are nonempty.\label{lem: pri-facts-max}
\end{enumerate}
\end{lemma}

\begin{remark} \label{rem: AC}
    Assuming PIT, 
    \cref{lem: pri-facts-max} is equivalent to AC
    (see \cite{AB25}).
\end{remark}

The next lemma will be used in \cref{sec: nuclei}.
\begin{lemma}[{\cite[Thm.~3]{Pri70}; see also \cite[Prop.~11]{Pri72}}] \label{lem: onto hom}
        A $\DLat$-morphism
    is onto if and only if its dual Priestley morphism is an order-embedding.
\end{lemma}

For a bounded distributive lattice $D$ and $S \subseteq D$, we write $$\sigma[S] := \{\sigma(s) \mid s \in S\},$$
and for a Priestley space $X$, we write $\opup(X)$ for the set
of open upsets.
The following proposition characterizes Priestley spaces of complete distributive lattices, where we write $\cl$ for the closure in a topological space. 

\begin{proposition}[{\cite[Props.~15 and 16]{Pri72}}]\label{prop: joins}
Let $D$ be a bounded distributive lattice and $X$ its Priestley space.
\begin{enumerate}[cref=proposition]
    \item\label{prop: joins-1} 
    For $S \subseteq D$, the join $\bigvee S$ exists in $D$ if and only if 
    $
        \upset \cl \bigcup \sigma[S]
    $
    is open \<equivalently, a clopen upset\>. In this case,
    \[
        \sigma\!\left(\bigvee S\right) = \upset \cl \bigcup \sigma[S].
    \]
    \item $D$ is complete if and only if $\upset \cl U$ is open for each $U \in \opup(X)$.
\end{enumerate}
\end{proposition}

\begin{remark}
A Priestley space $X$ is said to be \emph{extremally order disconnected} if 
$\upset \cl U$ is open for each $U \in \opup(X)$ 
(see \cite[p.~521]{Pri72}). Using this terminology,
a bounded distributive lattice $D$ is complete if and only if its Priestley space $X_D$ is extremally order disconnected. This generalizes the classic result (see, e.g., \cite[Ch.~38]{GH09}) that a boolean algebra is complete if and only if its Stone space is extremally disconnected.
\end{remark}

\begin{definition}
    Let $D$ be a 
    lattice.
    \begin{enumerate}
        \item An existing join 
        $\bigvee S$ in $D$ 
        is \emph{exact} if for each $a\in D$, the join $\bigvee\{a \wedge s \mid s \in S\}$ also exists and
        $
            a \wedge \bigvee S = \bigvee \{a \wedge s \mid s \in S\}.
        $
        \item $D$ is a \emph{frame} if it is complete and every join in $D$ is exact.
    \end{enumerate}
\end{definition}

\begin{remark}
    The terminology ``exact'' was introduced by Ball  \cite{Bal84}, but exact joins were already studied by MacNeille 
    \cite{Mac37} under the name of 
    ``distributive joins.'' 
\end{remark}

The following proposition, which goes back to Wigner \cite[Prop.~2]{Wig79} (see also \cite[Thm.~2.3]{PS88} and \cite[Thm.~5.16]{BDMWW25}), describes Priestley spaces of frames.

\begin{theorem}
Let $D$ be a bounded distributive lattice and $X$ its Priestley space.
\begin{enumerate}[cref=theorem]
    \item\label{prop: exact joins-1}
    An existing join $\bigvee S$ in $D$ is exact if and only if 
    \[
        \sigma\!\left(\bigvee S\right) = \cl \bigcup \sigma[S].
    \]
    \item\label{prop: exact joins-2}
    $D$ is a frame if and only if $\cl U \in \opup(X)$ for each $U\in\opup(X)$.
\end{enumerate}
\end{theorem}
\begin{proof}
    (1) By \cref{prop: joins-1}, it suffices to show that $\bigvee S$ is exact if and only if $\cl \bigcup \sigma[S]$ is an upset. 
    First suppose $\cl \bigcup \sigma[S]$ is an upset. Since $\sigma : D \to \clopup(X)$ is an isomorphism, by \cref{prop: joins-1} it suffices to show that
    \[
        V \cap \cl \bigcup \sigma[S] = \upset \cl\!\left(V \cap \bigcup \sigma[S]\right)
    \]
    for each $V \in \clopup(X)$. The inclusion $\supseteq$ is clear since $V \cap \bigcup \sigma[S] \subseteq V \cap \cl \bigcup \sigma[S]$ and $V \cap \cl \bigcup \sigma[S]$ is a closed upset. For the inclusion $\subseteq$, let $x \in V \cap \cl \bigcup \sigma[S]$ and let $U$ be an open neighborhood of $x$. Then $U \cap V$ is an open neighborhood of $x$, so $U \cap V \cap \bigcup \sigma[S] \ne \varnothing$. 
    Thus, $$x \in \cl\!\left(V \cap \bigcup \sigma[S]\right) \subseteq \upset \cl\left(V \cap \bigcup \sigma[S]\right).$$

    Conversely, suppose $\cl \bigcup \sigma[S]$ is not an upset. Then there exists $x \in \upset \cl \bigcup \sigma[S]$ with ${x \notin \cl \bigcup \sigma[S]}$. 
    Therefore, $\{x\} \cap \cl \bigcup \sigma[S] = \varnothing$. By \cref{lem: pri-facts},
    \[
        \{x\} = \upset x \cap \downset x = \bigcap \{U \in \clopup(X) \mid  x \in U \} \cap \bigcap \{ V^c \mid V \in \clopup(X) \text{ and } x \notin V\}.
    \]
    Since finite intersections and unions of clopen upsets are clopen upsets, by compactness there are $U, V \in \clopup(X)$ such that $x \in U \cap V^c$ and $U \cap V^c \cap \cl \bigcup \sigma[S] = \varnothing$. From $U, V \in \clopup(X)$ it follows that $U = \sigma(a)$ and $V = \sigma(b)$ for some $a,b \in D$. 
    Thus, $\sigma(a) \cap \cl \bigcup \sigma[S] \subseteq \sigma(b)$. If $\bigvee\{a \wedge s \mid s \in S\}$ does not exist in $D$, then $\bigvee S$ is not exact. Otherwise, by \cref{prop: joins-1},
    \[
        \sigma\Big(\bigvee_{s \in S} (a \wedge s)\Big) 
          = \upset \cl \bigcup_{s \in S} \bigl(\sigma(a) \cap \sigma(s)\bigr)
          = \upset \cl \bigl(\sigma(a) \cap \bigcup \sigma[S]\bigr)
          \subseteq \sigma(b).
    \]
    On the other hand,
    \[
        \sigma\left(a \wedge \bigvee S\right) = \sigma(a) \cap \upset \cl \bigcup \sigma[S] \not\subseteq \sigma(b)
    \]
    since $x \in \sigma(a) \cap \upset \cl \bigcup \sigma[S]$, but $x \notin \sigma(b)$. 
    Because $\sigma$ is an isomorphism, we conclude that $a \wedge \bigvee S \neq \bigvee\{a \wedge s \mid s \in S\}$, and hence $\bigvee S$ is not exact.  
    
    (2) This is immediate from (1).
\end{proof}

\begin{definition}[{see, e.g., \cite[Def.~3.1]{BM25}}]
 An \emph{\L-space} is a Priestley space $X$ such that ${\cl U \in \opup(X)}$ for each $U \in \opup(X)$.
\end{definition}

\begin{remark} \label{rem: names}
The prefix ``\L'' in the above definition refers to \emph{localic}. In \cite{PS88}, these spaces are called ``f-spaces'' (``f'' for frame), and in \cite{PS00} they are called ``LP-spaces'' (``localic Priestley spaces''). 
\end{remark}

As an immediate consequence of \cref{prop: exact joins-2}, we obtain: 

\begin{corollary}
    A bounded distributive lattice $D$ is a frame if and only if its Priestley space $X_D$ is an \L-space.
\end{corollary} 

\begin{remark}    
Every frame $L$ is a Heyting algebra (see, e.g., \cite[p.~11]{PP12}) since the implication
\[
    a \to b := \bigvee\{c \in L \mid a \wedge c \leq b\}
\]
exists for all $a,b \in L$. Hence, \L-spaces are \emph{Esakia spaces} (see, e.g., \cite{Esa19}).
\end{remark}

Recall (see, e.g., \cite[p.~10]{PP12}) that a \emph{frame homomorphism} is a bounded lattice morphism preserving arbitrary joins. On the Priestley side, this preservation translates into how the dual map interacts with closure.

\begin{proposition}[{\cite[Prop.~3]{Wig79}; see also \cite[\Sec 2.5]{PS88}}]\label{prop: morphisms}
Let $D$ and $E$ be complete distributive lattices, and let 
$h \colon D \to E$ be a bounded lattice morphism with the dual Priestley morphism 
$f \colon X_E \to X_D$.
\begin{enumerate}[cref=proposition]
    \item $h$ preserves arbitrary joins if and only if 
    \[
        f^{-1}(\upset \cl U) = \upset \cl f^{-1}(U)
    \]
    for each $U\in\opup(X_D)$.
    \item If $D$ and $E$ are frames, then $h$ is a frame homomorphism if and only if 
    \[
        f^{-1}(\cl U) = \cl f^{-1}(U)
    \]
    for each $U \in \opup(X_D)$.\label{prop: morphisms-2}
\end{enumerate}
\end{proposition}

\begin{proof} 
    (1) It is sufficient to show that 
    $f^{-1}\colon \clopup(X_D) \to \clopup(X_E)$ 
    preserves arbitrary joins if and only if
        $f^{-1}(\upset \cl U) = \upset \cl f^{-1}(U)$ for each $U \in \opup(X_D)$. By \cref{prop: joins-1}, $f^{-1}$ preserves arbitrary joins  exactly when
    \[
        f^{-1}\!\left(\upset \cl \bigcup \sigma[S]\right)
          = \upset \cl \bigcup_{s \in S} f^{-1}(\sigma(s))
          = \upset \cl f^{-1}\!\Bigl(\bigcup \sigma[S]\Bigr)
    \]
    for each $S \subseteq D$. Since every open upset of $X_D$ is of the form $\bigcup \sigma[S]$ (see \cref{lem: pri-facts-unions}), the latter is equivalent to $f^{-1}(\upset \cl U) = \upset \cl f^{-1}(U)$
    for each $U \in \opup(X_D)$.
    
    (2) The proof is the same as that for (1), but uses \cref{prop: exact joins-1} instead of \cref{prop: joins-1}, because of which the occurrences of $\upset$ become redundant.
\end{proof}

This motivates the following definition:

\begin{definition}[{see, e.g., \cite[Def.~3.1]{BM25}}]
An \emph{\L-morphism} $f : X \to Y$ is a Priestley morphism between \L-spaces such that $f^{-1}(\cl U) = \cl f^{-1}(U)$ for each $U \in\opup(Y)$.
\end{definition}

\begin{remark}
    In \cite{PS88}, \L-morphisms are called ``f-maps,'' and in \cite{PS00} they are called ``LP-maps'' (see \cref{rem: names} for the abbreviations). 
\end{remark}

It is easy to see that \L-spaces and \L-morphisms form a category, which we denote by \LPries. As is customary, we also let \Frm denote the category of frames and frame homomorphisms. By \cref{prop: exact joins-2,prop: morphisms-2}, the functors establishing Priestley duality restrict to the categories \Frm and \LPries to yield the following result (see \cite[Cor.~2.5]{PS88}):

\begin{theorem}[Pultr--Sichler--Wigner duality]
\label{thm: WPS duality}
    $\Frm$ is dually equivalent to $\LPries$.
\end{theorem}

\begin{remark} \label{rem: Wigner}
    In \cite{BM23,BM25}, we referred to the above  result as Pultr--Sichler duality. However, recently we came across Wigner's earlier work \cite{Wig79}. In view of this, it is more prudent to refer to this result as Pultr--Sichler--Wigner duality. 
\end{remark}

\section{Spatiality} \label{sec: spatiality}

Spatiality is one of the central notions in pointfree topology.
A frame $L$ is \emph{spatial} if it is isomorphic to the opens of a topological space. Equivalently, $L$ is spatial provided it has enough completely prime filters to separate its elements (see, e.g., \cite[p.~18]{PP12}). 
Spatial frames are precisely the fixed points of the 
dual adjunction between \Frm and the category \Top of topological spaces and continuous maps. This adjunction sends a topological space $X$ to its frame $\Omega(X)$ of opens and a continuous map $f : X \to Y$ to 
the frame homomorphism $f^{-1} : \Omega(Y) \to \Omega(X)$. Conversely, a frame $L$ is sent to its space $\pt(L)$ of \emph{points} (that is,  completely prime filters) and a frame homomorphism $h : L \to M$ to the continuus map $h^{-1} : \pt(M) \to \pt(L)$. 

Let $\SFrm$ be the full subcategory of \Frm consisting of spatial frames and let \Sob be the full subcategory of \Top consisting of sober spaces, where we recall that a topological space is  \emph{sober} if every closed irreducible set is the closure of a unique point. 

\begin{theorem}[{\cite[Thm.~8]{DP66}}] \label{thm: DP66}
    The functors $\Omega : \Top \to \Frm$ and $\pt : \Frm \to \Top$ are dually adjoint and restrict to a dual equivalence between $\Sob$ and $\SFrm$.
\end{theorem}

In this section we describe the Priestley spaces corresponding to spatial frames, which then will be used to recover the dual adjunction and equivalence of \cref{thm: DP66}.

\begin{lemma}[{\cite[Lem.~5.1]{BGJ16}}; see also {\cite[Prop.~2.9]{PS00}}]
    Let $L$ be a frame and $X$ its \L-space. Then  $x \in X$ is a completely prime filter if and only if $\downset x$ is open.
    \label{lem: localic part}
\end{lemma}
\begin{proof}
    ($\Rightarrow$) Suppose $x \in X$ is a completely prime filter. By \cref{lem: pri-facts-upclosed}, $\downset x$ is a closed downset, so $U := (\downset x)^c$ is an open upset. Let $S = \{a \in L \mid \sigma(a) \subseteq U\}$. By \cref{lem: pri-facts-unions}, $U = \bigcup \sigma[S]$. Since $L$ is a frame, every join in $L$ is exact. Therefore, by \cref{prop: exact joins-1}, $\cl U = \sigma(\bigvee S)$. Because $x \notin \sigma(s)$ for all $s \in S$ and $x$ is completely prime, $x \notin \sigma(\bigvee S) = \cl U$. Thus, $\cl U \cap \downset x = \varnothing$, so $\cl U = U$, and hence $\downset x$ is open.

    ($\Leftarrow$) Suppose $\downset x$ is open. If $\bigvee S \in x$, then $x \in \sigma(\bigvee S) = \cl \bigcup \sigma[S]$ by \cref{prop: exact joins-1}. Therefore, $\downset x \cap \sigma(a) \neq \varnothing$ for some $a \in S$, and since $\sigma(a)$ is an upset, $a \in x$. Thus, $x$ is a completely prime filter.
\end{proof}

The key definition of this section is that of the localic part of an \L-space.
\begin{definition}[{see, e.g., \cite[p.~231]{PS00}}]
    Let $X$ be an \L-space.
    \begin{enumerate}
        \item Call $x \in X$ a \emph{localic point} if $\downset x$ is clopen.
    \item The set $\loc X$ of all localic points of $X$ constitutes the \emph{localic part} of $X$.
    \end{enumerate}
\end{definition}

By \cref{lem: localic part}, spatiality of a frame translates into the requirement that its dual \L-space has enough localic points.
The latter has a very natural description: 

\begin{theorem}[{\cite[pp.~231--232]{PS00}; see also \cite[Thm.~5.5]{AB+20}}]\label{thm: spatial}
    Let $L$ be a frame and $X$ its \L-space.
    Then $L$ is spatial if and only if $\loc X$ is dense in $X$. 
\end{theorem}

\begin{proof}
($\Rightarrow$) Let $L$ be spatial and $U \subseteq X$ a nonempty open. By the definition of the topology on $X$ (see \cref{rem: Priestley functors}), there are $a,b \in L$ with $\varnothing\ne\sigma(a) \cap \sigma(b)^c \subseteq U$, so
$a \nleq b$. Since $L$ is spatial, there is $y \in \loc X$ with $a \in y$ and $b \notin y$. Thus, $y \in U$, and hence $\loc X$ is dense in $X$.  

($\Leftarrow$) If $a \nleq b$, then $\sigma(a) \cap \sigma(b)^c \neq \varnothing$. Since $\loc X$ is dense in $X$, there is ${y \in \loc X \cap \sigma(a) \cap \sigma(b)^c}$. Thus, $a \in y$ but $b \notin y$, and hence $L$ is spatial.
\end{proof}

This motivates the following definition.

\begin{definition}[{see, e.g., \cite[p.~232]{PS00}}]
    \leavevmode
    \begin{enumerate}
        \item An \L-space $X$ is \emph{\L-spatial} or a \emph{spatial \L-space} if $\loc X$ is dense in $X$.
        \item Let $\SLPries$ be the full subcategory of $\LPries$ consisting of spatial \L-spaces.
    \end{enumerate}
\end{definition}

Using \cref{thm: spatial}, we can restrict Pultr--Sichler--Wigner duality to spatial frames: 

\begin{corollary}[{\cite[Cor.~4.10]{BM23}}] \label{thm: spatial duality}
    \SFrm is dually equivalent to \SLPries.
\end{corollary}

Next, we describe how to derive the 
adjunction of \cref{thm: DP66} from the perspective of \L-spaces. For this we need to define an adjunction between \LPries and \Top.

\begin{definition}[{see, e.g., \cite[Def.~4.3]{BM23}}]
    Let $X$ be an \L-space. 
    We equip its localic part $\loc X$ with the topology 
    \[
        \{U \cap \loc X \mid U \in\opup(X)\}.
    \]
\end{definition}

\begin{remark} 
\leavevmode
\begin{enumerate}[cref=remark]
    \item\label{rem: loc topology}
    The opens of $\loc X$ can alternatively be described as the intersections of clopen upsets of $X$ with $\loc X$.
    Indeed, ${U \cap \loc X = \cl U \cap \loc X}$ for each $U \in\opup(X)$ (see, e.g., \cite[Lem.~4.4]{BM23}). 
    
    \item\label{rem: sober} If $X$ is the \L-space of a frame $L$, then $\loc X = \pt(L)$ by \cref{lem: localic part}. Moreover, the opens of $\pt(L)$ are of the form $\sigma(a)\cap\loc X$ for $a\in L$ (see, e.g., \cite[p.~15]{PP12}). Thus, $\loc X$ is homeomorphic to $\pt(L)$ (see \cite[Prop.~5.4]{AB+20}). Since the latter is a sober space, we conclude that so is $\loc X$. This will be used in later sections.  
\end{enumerate}
\end{remark}

\begin{lemma}[{\cite[Lem.~4.12]{BM23}}]
    If $f : X \to Y$ is an \L-morphism, then the restriction $f : \loc X \to \loc Y$ is a well-defined continuous map.
\end{lemma}

Thus, we obtain a functor $\Loc \colon \LPries \to \Top$ by sending 
an \L-space $X$ to $\loc X$ and an \L-morphism 
$f \colon X \to Y$ to its restriction 
$\loc f \colon \loc X \to \loc Y$.
    
Conversely, the functor $\Pri \colon \Top \to \LPries$ sends a topological space $Y$ to the Priestley space of its frame of opens $X_{\Omega(Y)}$ and a continuous map $f : Y \to Z$ to the \L-morphism $\widehat f \colon X_{\Omega(Y)} \to X_{\Omega(Z)}$ given by $\widehat f(y) = \{ U \in \Omega(Z) \mid f^{-1}(U) \in y\}$. 
If $y$ is a prime filter of $\Omega(Y)$, then $\widehat f(y)$ is a prime filter of $\Omega(Z)$, so $\widehat f$ is well defined.

\begin{theorem}[{\cite[Cor.~4.19]{BM23}}] \label{thm: adjunction in priestley}
    The functors $\Loc$ and $\Pri$ are adjoint and restrict to an equivalence between \SLPries and \Sob. 
\end{theorem}

\begin{remark}
\label{rem: cornish}
As mentioned in the introduction, \Pries is isomorphic to the category \Spec of spectral spaces (see also the end of 
\cref{subsec: coherent}). Under this isomorphism,
the functor $\Pri : \Top \to \LPries$ 
corresponds to 
sending a topological space to the spectral space of prime filters of its frame of opens. As is pointed out in \cite[Prop.~4]{Ban81}, 
this yields a reflection of \Top into \Spec. This reflection is identified in \cite[Prop.~16]{Smy92} as
the largest stable compactification of a $T_0$-space. Thus, the functor $\Pri$ can be seen as the Priestley analogue 
of Banaschewski’s reflection and Smyth’s largest stable compactification.
\end{remark}

Pultr--Sichler--Wigner duality together with \cref{thm: spatial duality,thm: adjunction in priestley} 
yield \cref{thm: DP66}.
The situation is summarized in \cref{fig: Frm LPries Top}, where hooked arrows denote being  a full subcategory, squiggly double arrows denote dual equivalences, straight double arrows denote equivalences, and curved arrows denote adjunctions. The labels on the arrows refer to the theorems and corollaries in which these results are established.

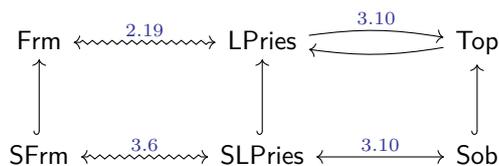
\begin{figure}[H]
\begin{center}
\begin{tikzcd}[column sep=huge, row sep=large]
\Frm \arrow[r, <->, dual, "\ref{thm: WPS duality}"] & \LPries \arrow[r, bend left=2ex, "\ref{thm: adjunction in priestley}"] & \Top \arrow[l, bend left=2ex] \\
\SFrm \arrow[u, hook] \arrow[r, <->, dual, "\ref{thm: spatial duality}"] & \SLPries \ar[r, <->, "\ref{thm: adjunction in priestley}"] \ar[u, hook] & \Sob \ar[u, hook]
\end{tikzcd}
\end{center}
\caption{Equivalences and adjunctions between frames, \L-spaces, and topological spaces, together with their spatial/sober restrictions.}
\label{fig: Frm LPries Top}
\end{figure}

\section{Sublocales and nuclei} \label{sec: nuclei}

Sublocales provide a generalization of subspaces and play a central role in pointfree topology 
(see, e.g., \cite[Ch.~III]{PP12}). We recall:
\begin{definition}
A subset $S$ of a frame $L$ is a \emph{sublocale} provided
\begin{enumerate}
    \item $S$ is closed under arbitrary meets;
    \item $a \to s \in L$ for each $a \in L$ and $s \in S$. 
\end{enumerate}
We denote by $\S(L)$ the set
of all sublocales of $L$.
\end{definition}

It is well known (see, e.g., \cite[p.~28]{PP12}) that $\S(L)$ ordered by inclusion is a coframe (i.e., the order dual of a frame). 
\begin{remark}
A long-standing open problem in pointfree topology, due to Isbell \cite{Isb72}, is to characterize which coframes arise this way.     
\end{remark}
 
Sublocales admit an equivalent description in terms of nuclei:
\begin{definition}[{see, e.g. \cite[p.~31]{PP12}}]
A \emph{nucleus} on a frame $L$ is a map $j \colon L \to L$ satisfying, for all $a,b \in L$,
\[
    a \leq j(a), \qquad j(j(a)) = j(a), \qquad j(a \wedge b) = j(a) \wedge j(b).
\]
We denote by $\N(L)$ the set
of all nuclei on $L$, ordered pointwise. 
\end{definition}

Given a nucleus $j \in \N(L)$, the set of its fixpoints $L_j := \{ a \in L \mid j(a) = a \}$ is a sublocale of $L$. Conversely, if $S \in \S(L)$, then $j_S(a) := \bigwedge \{ s \in S \mid a \leq s \}$
defines a nucleus on $L$ whose fixpoints are exactly $S$. 
This establishes a bijection between $\N(L)$ and $\S(L)$, and it is straightforward to verify that this bijection is a dual order-isomorphism. We thus arrive at the following well-known result:

\begin{theorem}[{see, e.g., \cite[pp.~31--32]{PP12}}] \label{thm: SL and NL}
    For a frame $L$, $\S(L)$ and $\N(L)$ are dually isomorphic.
\end{theorem}

Consequently, $\N(L)$ is a frame, known as the {\em assembly} of $L$ \cite[p.~242]{Sim78}. 
Sublocales (and hence nuclei) of a frame $L$ correspond to certain closed subsets of the Priestley space of $L$. 
\begin{definition}
    Let $X$ be an \L-space. 
    \begin{enumerate}
        \item We call a closed set $N \subseteq X$ a \emph{nuclear subset} of $X$ provided $\downset (U \cap N)$ is clopen for each clopen set $U \subseteq X$.

        \item We denote by $\N(X)$ the nuclear subsets of $X$, ordered by inclusion.
    \end{enumerate}
\end{definition}

The above definition originates in \cite{BG07}, where these sets are called ``subframes.'' The term ``nuclear'' was introduced in \cite{AB+20} because of their link to nuclei (see \cref{thm: nuclear-subsets}). The next lemma provides an alternate description of nuclear sets, which goes back to Wigner \cite[Prop.~5]{Wig79} (see also \cite[\Sec 2.4]{PS00}). 

\begin{lemma} \label{lem: Lsets}
    Let $X$ be an \L-space. A closed subset $N$ of $X$ is a nuclear subset if and only if 
    $N \cap \cl U = \cl (N \cap U)$
    for each $U \in\opup(X)$.
\end{lemma}

\begin{proof}
    ($\Rightarrow$) Let $U \in \opup(X)$. The inclusion 
    $\cl (N \cap U) \subseteq N \cap \cl U$ is clear. For the reverse inclusion, 
    let $x \in N \cap \cl U$, and let $V$ be a clopen neighborhood of $x$. 
    Since $N$ is nuclear, $\downset(N \cap V)$ is a clopen downset containing $x$, 
    so $\downset(N \cap V) \cap U \neq \varnothing$. Because $U$ is an upset, 
    $N \cap V \cap U \neq \varnothing$. Thus, $x \in \cl(N \cap U)$.

    ($\Leftarrow$) Let $U \subseteq X$ be clopen, and put $V = (\downset(N \cap U))^c$. 
    Then $V \in \opup(X)$ (see \cref{lem: pri-facts-upclosed}). 
    Suppose $V$ is not closed. Then
    $
        \cl V \cap V^c \neq \varnothing$, so $\cl V \cap \downset(N \cap U) \neq \varnothing.
    $
    Since $X$ is an \L-space, $\cl V$ is an upset. 
    Therefore, $\cl V \cap N \cap U \neq \varnothing$. 
    By assumption, $\cl V \cap N = \cl(V \cap N)$, so 
    $\cl(V \cap N) \cap U \ne \varnothing$, and hence $V \cap N \cap U \neq \varnothing$ since $U$ is open. But 
    $
        V \cap N \cap U = (\downset(N \cap U))^c \cap N \cap U = \varnothing,
    $
    a contradiction. 
    Thus, $V$ is clopen, yielding that so is $\downset (N \cap U)$. 
\end{proof}

The following theorem goes back to \cite[Prop.~5]{Wig79} (see also \cite[\Sec 2.4]{PS00}), where it was stated using the alternate description of nuclear subsets given in \cref{lem: Lsets}. The formulation below is 
from \cite[Thm.~30]{BG07}, but the proof we give 
is different (see \cref{rem: nucleus def}).

\begin{theorem}
\label{thm: nuclear-subsets}
Let $L$ be a frame and $X$ its \L-space. 
\begin{enumerate}
    \item $\N(X)$ and $\N(L)$ are dually isomorphic.
    \item $\N(X)$ and $\S(L)$ are isomorphic.
\end{enumerate}
\end{theorem}

\begin{proof}
    (1) 
    Define a map $\N(L) \to \N(X)$ by
    \[
        j \mapsto N_j := \{x \in X \mid x = j^{-1}[x]\}.
    \]
    To see that this is well defined, let $j \in \N(L)$. Then $j : L \to L_j$ is an onto frame homomorphism (see, e.g., \cite[p.~31]{PP12}). Letting $X_j$ be the \L-space of $L_j$, we have that $f := j^{-1}$ is a one-to-one \L-morphism from $X_j$ into $X$ (see \cref{lem: onto hom}).
    It is straightforward to verify that $f[X_j] = N_j$. Hence, 
    since $f$ is an \L-morphism, 
    $N_j \cap \cl U = \cl(N_j \cap U)$ for each $U \in\opup(X)$. Therefore, $N_j$ is a nuclear subset by \cref{lem: Lsets}. Thus, $j \mapsto N_j$ is well defined, and it is straightforward to see that $j_1 \le j_2$ implies $N_{j_2} \subseteq N_{j_1}$.
    
    Next define a map $\N(X) \to \N(L)$ by $N \mapsto j_N$ where $$j_N(a) = \bigvee\{b \in L \mid \sigma(b) \cap N \subseteq \sigma(a)\}.$$ It is straightforward to see that $a \leq j_N(a)$ for all $a \in L$. Also, since joins in $L$ are exact, $j_N(a \wedge b) = j_N(a) \wedge j_N(b)$ for all $a,b \in L$. 
    Moreover,
    \[\begin{aligned}
        \sigma(j_N(a)) \cap N 
        &= \cl \bigl( \medcup\{\sigma(b) \mid \sigma(b) \cap N \subseteq \sigma(a) \} \bigr) \cap N \\
        &= \cl \bigl( \medcup\{\sigma(b) \mid \sigma(b) \cap N \subseteq \sigma(a) \} \cap N \bigr) \\
        &= \cl \medcup\{\sigma(b) \cap N \mid \sigma(b) \cap N \subseteq \sigma(a)\} \\
        &= \sigma(a) \cap N,
    \end{aligned} \tag{A}\label{eq:A}\]
    where the second equality holds by \cref{lem: Lsets}. 
    This implies that $\sigma(b) \cap N \subseteq \sigma(j_N(a))$ if and only if $\sigma(b) \cap N \subseteq \sigma(a)$, and hence $j_N(j_N(a))=j_N(a)$ for all $a\in L$.
    Thus, $N \mapsto j_N$ is well defined, and it is straightforward to see that $N_1 \subseteq N_2$ implies $j_{N_2} \le j_{N_1}$.
    
    For $a,b \in L$, we have
    \begin{align*}
        b \leq j(a) 
        &\iff (\forall x \in X)(b \in x \text{ and }j^{-1}[x] = x \Longrightarrow 
        a \in x) \\
        &\iff \sigma(b) \cap N_j \subseteq \sigma(a)\\
        &\iff b \leq j_{N_j}(a),
    \end{align*}
    where in the first equivalence the right-to-left direction follows from PIT (since if $b \nleq j(a)$ then $j(b) \nleq j(a)$, so PIT applied to $L_j$ gives $x \in X_j$ with $j(b) \in x$ and $j(a) \notin x$, and then $j^{-1}[x]$ is the desired point of $X$ containing $b$ and missing $a$).
    Also, for $x \in X$,
    \begin{align*}
        x \in N_{j_N}
        &\iff j_N^{-1}[x] \subseteq x \\
        &\iff (\forall a \in L)(j_N(a) \in x \Longrightarrow
        a \in x) \\
        &\iff (\forall a \in L)\left(x \in \sigma(j_N(a)) \Longrightarrow
        x \in \sigma(a)\right) \\
        & \iff x \in N,
    \end{align*}
    where the right-to-left implication of the last equivalence follows from \cref{eq:A}; 
to see the left-to-right implication,
if $x \notin N$, then since $N$ is closed, there are $a,b \in L$ with 
$x \in \sigma(b) \cap \sigma(a)^c$ and $\sigma(b) \cap \sigma(a)^c \cap N = \varnothing$. 
Thus, $x \in \sigma(b)$ and $\sigma(b)\cap N \subseteq \sigma(a)$, so $x \in \sigma(j_N(a))$ but $x \notin \sigma(a)$. 

    (2) This follows from (1) and \cref{thm: SL and NL}.
\end{proof}

\begin{remark}\label{rem: nucleus def}
    In \cite{BG07}, the map $\N(X) \to \N(L)$ is defined by first introducing the nucleus $j_N'$ on $\clopup(X)$ by setting
    \[
        j_N'(U) = X \setminus \downset (N \setminus U) \qquad \text{for } U \in \clopup(X),
    \]
    and then transporting it along the isomorphism $\sigma \colon L \to \clopup(X)$ to obtain a nucleus on $\N(L)$. 
    Our proof of \cref{thm: nuclear-subsets} provides an alternative description of this nucleus. Indeed, if $a \in L$, then for $U \in \clopup(X)$ we have 
    \begin{align*}
        U \subseteq \sigma(j_N(a)) 
        &\iff U \subseteq \cl \medcup \{\sigma(b) \mid \sigma(b) \cap N \subseteq \sigma(a)\} \\
        &\iff U \cap N \subseteq \sigma(a)  &&\text{(by \cref{lem: Lsets})}\\
        &\iff N \cap \sigma(a)^c \subseteq U^c \\ 
        &\iff \downset\!\big(N \cap \sigma(a)^c\big) \subseteq U^c \\ 
        &\iff U \subseteq j'_N(\sigma(a)).
    \end{align*}
    Hence, $\sigma \circ j_N = j_N' \circ \sigma$.
\end{remark}

As a consequence of \cref{thm: nuclear-subsets}, we obtain that $\N(X)$ is a coframe. We next describe how joins and meets are calculated in $\N(X)$.

\begin{theorem}[{\cite[Lem.~4.8 and Thm.~4.9]{AB+20}}] \label{theorem: nuclear meets and joins}
    Let $X$ be an \L-space. For a family $\{ N_i \} \subseteq \N(X)$, we have
    \[
        \bigvee N_i = \cl \bigcup N_i.
    \]
    Consequently, 
    \[
        \bigwedge N_i = \cl \bigcup \left\{N \in \N(X) \mid N \subseteq \bigcap N_i\right\}.
    \]
\end{theorem}

In particular, \cref{theorem: nuclear meets and joins} implies that a finite join of nuclear subsets is their set-theoretic union.
Certain sublocales form  building blocks of $\S(L)$ in the sense that each sublocale is determined by them. 

\begin{definition}[see, e.g., {\cite[pp.~33, 42--43]{PP12}}]
    Let $L$ be a frame and $a \in L$. Consider the nuclei defined by
    \[
        o_a(d) = a \to d, \qquad 
        c_a(d) = a \vee d, \qquad 
        b_a(d) = (d \to a) \to a 
    \]
    for $d \in L$. We denote the images  of these nuclei by $\o(a)$, $\c(a)$, and $\b(a)$, respectively. 
\end{definition}

\begin{remark}
    The nuclei $o_a, c_a, b_a$ were first considered in \cite{Sim78,BM79,Mac81},  where the notation $v_a, u_a, w_a$ was used. 
\end{remark}

The sublocales $\o(a)$, $\c(a)$, and $\b(a)$ are called the \emph{open}, \emph{closed}, and \emph{relatively dense} sublocales determined by $a$, and we use the same terminology for their associated nuclei. Observe that $b_0$ is the well-known \emph{double negation} nucleus $d \mapsto d^{**}$, where $d^* := d\to 0$ denotes the pseudocomplement of $d$. 
By the Isbell Density Theorem \cite[Thm.~1.5]{Isb72}, $\b(0)$ is the least dense sublocale. More generally, $\b(a)$ is the least sublocale containing $a$. We will discuss this in more detail towards the the end of the section. 
The next theorem characterizes open, closed, and relatively dense sublocales in the language of $\L$-spaces. 

\begin{theorem}[{\cite[Lem.~2.5]{PS00} and \cite[Thm.~34]{BG07}}] \label{thm: desc nuclei}
    Let $L$ be a frame and $X$ its \L-space. Let also $S$ be a sublocale of $L$ and $N$ its corresponding nuclear subset of $X$.
\begin{enumerate}[cref=theorem]
    \item $S = \o(a)$ if and only if $N = \sigma(a)$. \label{thm: desc open sublocale}
    \item $S = \c(a)$ if and only if $N = \sigma(a)^c$. \label{thm: desc closed sublocale}
    \item $S = \b(a)$ if and only if $N = \max (\sigma(a)^c)$. \label{thm: rel dense}
\end{enumerate} 
\end{theorem}

\Cref{thm: desc nuclei} shows that open sublocales correspond to clopen upsets, 
closed sublocales to clopen downsets, and relatively dense sublocales to the maxima of clopen downsets.  Therefore, clopen upsets and clopen downsets ar nuclear subsets, and hence $$\{ U \cup V^c \mid U,V \in \clopup(X) \} \subseteq \N(X)$$ by \cref{theorem: nuclear meets and joins}. In addition, $$\{ \max(U^c) \mid U \in \clopup(X) \} \subseteq \N(X)$$ by \Cref{thm: rel dense}.
We now describe how to obtain all nuclear subsets from the above ones.

\begin{proposition}
    Let $X$ be an \L-space and $N \in \N(X)$.
    \begin{enumerate}[cref=proposition]
        \item $N = \bigwedge \{U \cup V^c \mid N \subseteq U \cup V^c \text{ and } U,V \in \clopup(X)\}$. \label{prop: blocks-1}
        \item $N = \bigvee 
        \{\max(U^c) \mid \max (U^c) \subseteq N \text{ and } U \in \clopup(X)\}.$
    \end{enumerate}
\end{proposition}
\begin{proof} 
    (1) Let $\mathcal F = \{U \cup V^c \mid N \subseteq U \cup V^c
    \text{ and } U ,V \in \clopup(X)\}$. 
    By \cref{theorem: nuclear meets and joins}, $N \subseteq \bigwedge \mathcal F$. For the other inclusion, we use \cref{theorem: nuclear meets and joins} again, by which it suffices to show that for each nuclear subset $M \subseteq \bigcap \mathcal F$, we have $M \subseteq N$. Suppose not.
    Then there is $x \in M \setminus N$. Since $N$ is closed and $\{ V \cap U^c \mid U,V \in \clopup(X) \}$ is a basis for the topology, there are $U,V \in \clopup(X)$ with $x \in V \cap U^c$ and $V \cap U^c \cap N = \varnothing$. 
    Thus, $N \subseteq U \cup V^c$, so $U\cup V^c \in \mathcal F$, and hence $x \in M \subseteq U \cup V^c$, a contradiction.
    
    (2) The right-to-left
    inclusion is clear. For the left-to-right
    inclusion, let $x \in N$ and $V$ be a clopen neighborhood of $x$. 
    Since $N$ is nuclear, $D := \downset (N \cap V)$ is a clopen downset. Moreover, $\max D = \max (N \cap V) \subseteq N$. 
    Let $U = D^c$. Then $U \in \clopup(X)$ and ${\max (U^c) = \max D \subseteq N}$. 
    It is also clear that $V \cap \max(U^c) \ne \varnothing$ (see \cref{lem: pri-facts-max}). 
    Therefore, 
    $$x \in \cl \bigcup \{\max(U^c) \mid \max(U^c) \subseteq N \text{ and } U \in \clopup(X) \}$$ and \cref{theorem: nuclear meets and joins} completes the proof.
\end{proof}

As an immediate consequence, we obtain:

\begin{corollary}[]
    \leavevmode
    \begin{enumerate}
        \item \textup{(see, e.g., \cite[p.~36]{PP12})} Every sublocale is a meet of binary joins of open and closed sublocales.
        \item \textup{(see, e.g., \cite[p.~44]{PP12})} Every sublocale is a join of relatively dense sublocales. 
    \end{enumerate}
    The dual statements hold for nuclei \textup{(see, e.g., \cite[Lem.~7]{Sim78})}.
\end{corollary}

Recall that 
a frame is 
\emph{zero-dimensional} if every element is a join of complemented elements.

\begin{theorem}[{\cite[Thm.~32]{BG07}; see also \cite[Rem.~4.15]{BGJ13}}]\label{thm: center}
    Let $X$ be an \L-space. Each clopen subset of $X$ is complemented in 
    $\N(X)^{\mathrm{op}}$.
        \label{thm: center-1}
\end{theorem}

Putting \cref{prop: blocks-1,thm: center-1} together, we obtain the following well-known result:

\begin{corollary}[see, e.g.,  {\cite[p.~105]{PP12}}]
    For a frame $L$, $\N(L)$ is zero-dimensional.
\end{corollary}

Recall (see, e.g., \cite[p.~334]{PP12}) that the \emph{booleanization} $\B(L)$ of a frame $L$ is 
the fixpoints of the double negation nucleus. In other words, $\B(L) = \b(0)$. By the well-known Isbell Density Theorem (see, e.g., \cite[p.~40]{PP12}), $\B(L)$ is the least dense sublocale of $L$. This result can be proved using Priestley duality (see \cite[Sec.~3]{BBM25}). 
To this end, we recall the following definition.

\begin{definition}[{see, e.g., \cite[p.~39]{PP12}}]
    Let $L$ be a frame. A sublocale $S \in \S(L)$ is \emph{dense} if $0 \in S$, and a nucleus $j \in \N(L)$ is \emph{dense} if $j(0)=0$.
\end{definition}

The following is a generalization of \cite[Lem.~3.16]{BBM25}.

\begin{lemma}
\label{lem: density}
    Let $L$ be a frame, $X$ its \L-space, $S \in \S(L)$, and $a \in L$. The following are equivalent.
    \begin{enumerate}[cref=lemma]
        \item $a \in S$.
        \item $j_S(a) = a$
        \item $\max (\sigma(a)^c) \subseteq N_{j_S}$. \label{lem: dense a}
    \end{enumerate}
    In particular, $S$ is dense if and only if $j_S$ is dense, which is equivalent to 
    $\max X \subseteq N_{j_S}$.
\end{lemma}
\begin{proof}
    (1)$\Leftrightarrow$(2) is obvious. For (2)$\Leftrightarrow$(3), let $j \in \N(L)$ and $N \in \N(X)$ be its corresponding nuclear subset. For $a \in L$, we have
    \begin{align*}
        j(a) = a 
        &\iff j_{N}'(\sigma(a)) = \sigma(a) &&\text{(see \cref{rem: nucleus def})} \\
        &\iff \downset (N \cap \sigma(a)^c) = \sigma(a)^c \\
        &\iff \max (\sigma(a)^c) \subseteq N, 
    \end{align*}
    where the last equivalence is a consequence of $\sigma(a)^c = \downset \max(\sigma(a)^c)$ (see \cref{lem: pri-facts-max}). 
    The result follows.
\end{proof}

\begin{theorem}\textup{(see, e.g., \cite[p.~42]{PP12})}
    For a frame $L$ and $a \in L$, $\b(a)$ is the least sublocale containing $a$.
\end{theorem}
\begin{proof}
    Let $S \in \S(L)$ be such that $a \in S$. By \cref{lem: dense a}, $\max(\sigma(a)^c) \subseteq N_{j_S}$. Therefore, by \cref{thm: nuclear-subsets,thm: rel dense}, $\b(a) \subseteq S$.
\end{proof}

Recalling that $\B(L) = \b(0)$, as an immediate corollary, we obtain Isbell’s Density Theorem:

\begin{corollary}[{\cite[Thm.~1.5]{Isb72}}]
$\B(L)$ is the least dense sublocale of~$L$.
\end{corollary}

The booleanization is an important frame-theoretic construction that does not have a spatial analogue since a topological space need not have a least dense subspace. It is of interest that the booleanization of the assembly of a frame has a convenient description using Priestley duality:

\begin{theorem}[{\cite[Prop.~4.12]{BGJ13}}]
Let $L$ be a frame and $X$ its \L-space. Then the booleanization of $N(L)$ 
is isomorphic to the 
regular closed subsets of $X$.
\end{theorem}

\section{Separation axioms}\label{sec: sep axioms}

Separation axioms play an important role in topology 
as they add extra structure, producing subcategories that allow deeper analysis. 
Since higher separation axioms (such as regularity or normality)
can be expressed purely in terms of the frame of opens, they can be described 
faithfully in the pointfree setting. 
On the other hand, lower separation axioms (such as $T_1$ or $T_2$)
are harder to capture pointfree because the frame language is too restrictive to express, for example, that points are closed. 
Nevertheless, several frame-theoretic analogs of $T_1$ and $T_2$  
have been introduced (see, e.g., \cite[Ch.~II--IV]{PP21}). Here we focus on two of the most prominent ones:
Isbell's subfitness (\cite{Isb72}) and the notion of Hausdorff frames introduced by Johnstone and Sun \cite{JS88}.

Separation axioms in pointfree topology have been studied through the lenses of Priestley duality (see, e.g., \cite{BCM23,BMMTWW25} for subfitness and  \cite{PS88,BGJ16} for regularity).
We provide a uniform approach 
by identifying certain \emph{kernels}, i.e., dense subsets of clopen upsets that determine the corresponding separation property. This follows the viewpoint of our recent work
\cite{BM23,BM25} (see also \cite[pp.~51, 108--109]{Mel25}), where similar kernels are introduced to characterize certain classes of frames (see \cref{sec: compact,sec: algebraic}).
We begin with lower separation axioms (subfitness and Hausdorffness) and then 
proceed to higher ones (regularity and normality).

\subsection{Subfitness}

As we pointed out above, subfitness was introduced as a frame-theoretic analog of the $T_1$ separation axiom (see, e.g., \cite{PP21}).

\begin{definition}
     A frame $L$ is \emph{subfit} 
     if for all $a,b \in L$,
     \[
     a \nleq b \Longrightarrow \exists c \in L : a \vee c =1 \mbox{ and } b \vee c \neq 1.
     \]
\end{definition}

This notion goes back to Wallman \cite[p.~115]{Wal38}.
The term  was coined by Isbell \cite[p.~17]{Isb72},\footnote{Subfit frames are also called \emph{disjunctive} (see, e.g., \cite{Sim79}), in correspondence with Wallman's original terminology for the dual notion.} who gave the following characterization of subfitness in the language of sublocales.

\begin{lemma}[see, e.g., {\cite[p.~29]{PP21}}]
A frame $L$ is subfit if and only if every open sublocale is a join of closed ones.
\label{lem: Isbell char subfit}
\end{lemma}

\L-spaces of subfit frames were described in \cite[Lem.~3.1]{BCM23} (see also \cite[Prop.~5.2]{BMMTWW25}) as those whose set of minimal points is dense. Below we give several other equivalent conditions, including the one using subfit kernels: 
 
\begin{definition}
    Let $X$ be an \L-space. For $U \in \clopup(X)$, define the \emph{subfit kernel} of $U$ as $\sfit U = \{x \in U \mid \downset x \subseteq U\}$.
\end{definition}

\begin{theorem}  \label{thm: subfit}
    Let $L$ be a frame and $X$ its \L-space. The following are equivalent.
    \begin{enumerate}[cref=theorem]
        \item $L$ is subfit. \label{thm: subfit-1}
        \item $\cl (\min U) = U$ for each $U \in \clopup(X)$. \label{thm: subfit-2}
        \item $\min X$ is dense. \label{thm: subfit-3}
        \item $\min X \subseteq N$ implies $N = X$ for each $N \in \N(X)$. \label{thm: subfit-4}
        \item $\cl (\sfit U) = U$ for each $U \in \clopup(X)$. \label{thm: subfit-5}
    \end{enumerate}
\end{theorem}
\begin{proof}
    (1)$\Rightarrow$(2) 
    Let $U \in \clopup(X)$. If $\cl(\min U) \neq U$, then 
    $U \setminus \cl(\min U) \neq \varnothing$, so there are $V,W \in \clopup(X)$ such that 
    $$\varnothing \neq V \cap W^c \subseteq U \setminus \cl(\min U).$$ 
    Therefore, $V \not\subseteq W$. Since $L \cong \clopup(X)$,  by (1) there is ${A \in \clopup(X)}$ such that 
    ${A \cup V = X}$ and $A \cup W \neq X$. This implies that $A^c \subseteq V$ and $A^c \not \subseteq W$. 
    In particular, $\min (A^c) \subseteq V$ and $\min (A^c) \not \subseteq W$ (by \cref{lem: pri-facts-max}). 
    Thus, there exists $x \in \min (A^c)$ with $x \notin W$, and so ${x \in V \cap W^c \subseteq U \setminus \cl(\min U)}$. On the other hand, since $A^c$ is a downset, $\min(A^c)\subseteq V \cap \min X$, yielding that $x \in U \cap \min X \subseteq \min U$, 
    so $x \in \cl(\min U)$, a contradiction. Hence, $\cl(\min U) = U$.

    (2)$\Rightarrow$(3) 
    Since $X \in \clopup(X)$, $\cl(\min X) = X$ by  (2). Thus, $\min X$ is dense.

    (3)$\Rightarrow$(4) Suppose $\min X \subseteq N$. Since $N$ is closed, $\cl(\min X) \subseteq N$. Thus, $N = X$ by (3).

    (4)$\Rightarrow$(5). Clearly $\cl(\sfit U) \subseteq U$. Let $U \not\subseteq \cl(\sfit U)$. Then there is $x \in U$ with ${x \notin \cl(\sfit U)}$. Because $X$ is a Stone space (see \cref{lem: pri-facts-stone}), there is a clopen neighborhood $V$ of $x$ with ${V \cap \sfit U = \varnothing}$. 
    Since $U \cap \min X \subseteq \sfit U$, we have ${U \cap V \cap \min X = \varnothing}$, so $\min X \subseteq (U \cap V)^c$. By \cref{thm: center-1}, $(U \cap V)^c \in \N(X)$, so (4) gives that $(U \cap V)^c = X$. This contradicts that $x \in U \cap V$. Thus, $U = \cl(\sfit U)$.

    (5)$\Rightarrow$(1) It suffices to show that $\clopup(X)$ is subfit. 
    Let $A \not \subseteq B$ for some $A,B \in \clopup(X)$. Then $A \cap B^c \neq \varnothing$, so 
    $\sfit(A) \cap B^c \neq \varnothing$ by (5). 
    Therefore, there exists $x \in B^c$ such that $\downset x \subseteq A$. 
    Since $B^c$ is a downset, $\downset x \subseteq A \cap B^c$. But $\downset x = \bigcap \{C^c \mid C \in \clopup(X) \text{ and } \downset x \subseteq C^c\}$ by \cref{lem: pri-facts-intersection}.
    Thus, 
    $
        \bigcap \{C^c \mid C \in \clopup(X) \text{ and } \downset x \subseteq C^c\} \subseteq A \cap B^c,
    $
    and by compactness, there is $C \in \clopup(X)$ such that $C^c \subseteq A \cap B^c$. 
    Consequently, $A \cup C = X$ and $B \cup C \neq X$ (since $B \cup C = C$ and $x \notin C$). 
\end{proof}

\begin{remark}
\Cref{thm: subfit} provides an alternate proof, in the language of Priestley spaces,  of 
several known frame-theoretic characterizations of subfitness:
\begin{itemize}
    \item Since clopen upsets correspond to open sublocales (see \cref{thm: desc open sublocale}), \cref{thm: subfit-2} states that every open sublocale 
    is subfit. This is equivalent to subfitness since a frame is an open sublocale of itself and open sublocales of subfit frames are subfit (see, e.g., \cite[p.~91]{PP12}).
    \item Because $\min X \subseteq N$ means that $N$ meets every nonempty clopen downset and the latter correspond to closed sublocales (see \cref{thm: desc closed sublocale}),  \cref{thm: subfit-4} states that the only sublocale that is not disjoint from any nonempty closed sublocale is the frame itself \cite[p.~29]{PP21}. 
    \item Since $\downset x$ is a closed downset, compactness together with \cref{lem: pri-facts-intersection} yields that $\downset x \subseteq U$ for $U \in \clopup(X)$ if and only if there is a clopen downset $D$ such that $x \in D \subseteq U$. Hence, \cref{thm: subfit-5} states that every open sublocale (clopen upset) is a join of closed sublocales (clopen downsets), which is Isbell's characterization of subfitness (see \cref{lem: Isbell char subfit}).
\end{itemize}
\end{remark}

\subsection{Hausdorffness}

Hausdorffness is another separation axiom 
that is difficult to express in the language of frames. 
Several pointfree definitions 
have been proposed in the literature (see, e.g., \cite[Ch.~III]{PP21}). We concentrate on the one introduced by Johnstone and Sun \cite{JS88}.

\begin{definition}[{\cite[p.~44]{PP21}}]
A frame $L$ is \emph{Hausdorff} if for all $a,b \in L$, $$a \neq 1 \mbox{ and } a\nleq b \Longrightarrow \exists c \in L : c \nleq a \mbox{ and } c^* \nleq b.$$
\end{definition}

While this notion does not imply subfitness,
it has become the accepted one as it agrees with the usual Hausdorffness in the spatial case and is preserved by important operations on frames such as taking sublocales and products (see, e.g., \cite[p.~44--46]{PP21}). 
There are several equivalent characterizations of Hausdorff frames, which will be used below.

\begin{lemma}[see, e.g., {\cite[pp.~43--44]{PP21}}] 
    A frame $L$ is Hausdorff if and only if 
        \[u = \bigvee \{v \in L \mid v \leq u \text{ and } v^* \nleq u\}\] 
        for all $u \in L \setminus\{1\}$. \label{lem: T2-join}
\end{lemma}

We now characterize Priestley spaces of Hausdorff frames using an appropriate kernel.

\begin{definition}
    Let $X$ be an \L-space. For $U \in \clopup(X)$, define the \emph{Hausdorff kernel} of $U$ as 
    $\haus U = \{x \in U \mid U \ne X \Longrightarrow \exists y \in \max (U^c) : \upset y \cap \upset x = \varnothing\}$.
\end{definition}

We next prove that a frame is Hausdorff iff 
the Hausdorff kernel of each clopen upset is dense in the upset.
To do so, we require the following lemma, which describes the pseudocomplement of a clopen upset (and the containment relationship). 

\begin{lemma}
    Let $X$ be an \L-space and $U,V \in\clopup(X)$. 
    \begin{enumerate}[cref=lemma]
        \item $U^* = X \setminus \downset U$. \label{lem: pseudocomplement-1}
        \item $V^* \subseteq U$ if and only if $\max(U^c) \subseteq \downset V$. \label{lem: pseudocomplement}
    \end{enumerate}
\end{lemma}

\begin{proof}
(1) This is well known (see, e.g., \cite[p.~61]{Esa19}). For 
    (2), observe that \begin{equation*}
        V^* \subseteq U 
        \iff X \setminus \downset V \subseteq U
        \iff U^c \subseteq \downset V 
        \iff \max(U^c) \subseteq \downset V. \qedhere
    \end{equation*}
\end{proof}

\begin{theorem} \label{thm: char Hausdorff} Let $L$ be a frame and $X$ its \L-space. 
    Then $L$ is Hausdorff if and only if $\cl (\haus U) = U$ for each clopen upset $U \in \clopup(X)$.
\end{theorem}
\begin{proof}
    Since $L \cong \clopup(X)$, by \cref{lem: T2-join} $L$ is Hausdorff if and only if
    \[
        U = \bigvee \{V \in \clopup(X) \mid V \subseteq U \text{ and } V^* \nsubseteq U\}
    \]
    for all $U \in \clopup(X) \setminus \{X\}$. 
    Using \cref{lem: pseudocomplement,prop: exact joins-1}, this is equivalent to \[
        U = \cl \bigcup \{V \in \clopup(X) \mid V \subseteq U \text{ and } \max(U^c) \nsubseteq \downset V\}.
    \]
    for all $U \in \clopup(X) \setminus \{X\}$. It remains to show that for such $U$,
    \[
      \haus U = \bigcup \{V \in \clopup(X) \mid V \subseteq U \text{ and } \max(U^c) \nsubseteq \downset V\}.
    \]
    ($\subseteq$) Let $x \in \haus U$. Then there is $y \in \max(U^c)$ such that $\upset x \cap \upset y = \varnothing$. By \cref{lem: pri-facts-intersection} and compactness, there is $V \in \clopup(X)$ such that $x \in V$ and $V \cap \upset y = \varnothing$. Since $y \notin U$, we may choose $V$ so that $V \subseteq U$. 
    From $V \cap \upset y = \varnothing$ it follows that $y \notin \downset V$. Thus, $\max(U^c) \not\subseteq \downset V$.

    ($\supseteq$) Let $x \in V \in \clopup(X)$ with $V \subseteq U$ and $\max(U^c) \nsubseteq \downset V$. Then there is $y \in \max(U^c)$ with $\upset y \cap V = \varnothing$. Since $x \in V$, we have $\upset y \cap \upset x = \varnothing$. 
    Thus, $x \in \haus U$. 
\end{proof}

Because Hausdorff frames need not be subfit, the Hausdorff kernel is not generally contained in the subfit kernel. We next define an alternative kernel that characterizes subfit Hausdorff frames and leads to a chain of containments between the kernels (see \cref{rem: containment of kernels}).

\begin{definition}
    Let $X$ be an \L-space. For $U \in \clopup(X)$, let
    \begin{align*}
        \hausfit U &= \haus U \cap \sfit U \\
        &=  \{x \in U \mid \downset x \subseteq U \text{ and } U \ne X \Longrightarrow \exists y \in \max (U^c) : \upset y \cap \upset x = \varnothing\}.
    \end{align*}
\end{definition}

\begin{theorem} \label{thm: char subfit Hausdorff} Let $L$ be a frame and $X$ its \L-space. 
Then $L$ is subfit Hausdorff if and only if $\cl(\hausfit U) = U$ for each clopen upset~$U$.
\end{theorem}
\begin{proof}
    ($\Rightarrow$) 
    It is sufficient to show that for each $x \in U$ and clopen neighborhood $W$ of $x$, we have ${W \cap \hausfit U \neq \varnothing}$. 
    Since $L$ is subfit, $\min X$ is dense by \cref{thm: subfit-3}. 
    Therefore, ${W \cap U \cap \min X \neq \varnothing}$, so there is $m \in \min X \cap W \cap U$. If $U = X$, then $m \in \hausfit U$, so $W \cap \hausfit U \ne \varnothing$. Suppose $U \ne X$. We have $\downset m = \{m\} \subseteq W \cap U$, so \cref{lem: pri-facts-intersection} and compactness imply that there is a nonempty clopen downset $D$ such that $D \subseteq W \cap U$. Because $L$ is Hausdorff, $D \cap  \haus U \neq \varnothing$, so there are $y\in U \cap D$ and $z \in \max(U^c)$ such that $\upset y \cap \upset z = \varnothing$. Therefore, $\downset y \subseteq D \subseteq U$, 
    yielding that $y \in W \cap \hausfit U$,
    as required. 
    
    ($\Leftarrow$) 
    Since $\hausfit U \subseteq \sfit U$, we have that $\sfit U$ is dense in $U$, and hence $L$ is subfit by \cref{thm: subfit-5}. 
    Also, since $\hausfit U \subseteq \haus U$, we have that $\haus U$ is dense in $U$, and hence $L$ is Hausdorff by \cref{thm: char Hausdorff}.
\end{proof}

\subsection{Regularity and complete regularity}
In contrast to lower separation axioms, higher separation axioms have more faithful pointfree descriptions.

\begin{definition}[{see, e.g., \cite[p.~88--89]{PP12}}]
    Let $L$ be a frame.
    \begin{enumerate}
        \item Let $a, b \in L$. We say that $a$ is \emph{rather below} $b$ and write $a \prec b$ if $a^* \vee b =1$.\footnote{This relation is also called the {\em well inside} relation (see, e.g., \cite[p.~80]{Joh82}).} 
        \item $L$ is \emph{regular} if for all $a,b \in L$,
        \[
            a \nleq b \implies \exists c \in L : c \prec a \text{ and } c \nleq b.
        \]
        Equivalently, $L$ is regular provided
        $a = \bigvee \{c \in L \mid c \prec a\}$ for all $a \in L$. 
    \end{enumerate}
\end{definition}

The Priestley spaces of regular frames have been described in \cite[Thm.~3.4]{PS88} and \cite[Lem.~3.6]{BGJ16}. Our approach follows the latter.

\begin{lemma}[{\cite[Lem.~3.3]{BGJ16}}]
    Let $X$ be an \L-space. For $U, V \in \clopup(X)$, 
         $U \prec V$ if and only if $\downset U \subseteq V$. \label{lem: prec}
\end{lemma}
\begin{proof}
    By \cref{lem: pseudocomplement-1},
    \[
       U \prec V \iff U^* \cup V = X \iff (X \setminus \downset U) \cup V = X \iff \downset U \subseteq V. \qedhere
    \]
\end{proof}

\begin{definition}[{see, e.g., \cite[Def.~ 3.4]{BGJ16}}]
    Let $X$ be an \L-space. For $U \in \clopup(X)$, define the \emph{regular kernel} of $U$ as $\reg U = \{x \in U \mid \downset \upset x \subseteq U\}$.
\end{definition}

\begin{theorem}[{\cite[Lem.~3.6]{BGJ16}}] \label{thm: char regular}
    Let $L$ be a frame and $X$ its \L-space. Then 
    $L$ is regular if and only if $\cl (\reg U) = U$ for each $U \in \clopup(X)$.
\end{theorem}
\begin{proof}
    Since $L \cong \clopup(X)$, by \cref{lem: prec,prop: exact joins-1} $L$ is regular if and only if
    \[
        U = \cl \bigcup \{V \in \clopup(X) \mid \downset V \subseteq U\}
    \]
    for all $U \in \clopup(X)$. Thus, it suffices to show that 
    \[
        \reg U = \bigcup \{V \in \clopup(X) \mid \downset V \subseteq U\}.
    \]
    
    ($\supseteq$) Let $x \in V \in \clopup(X)$ with $\downset V \subseteq U$. Then $\upset x \subseteq V$, so ${\downset \upset x \subseteq \downset V \subseteq U}$.
    
    ($\subseteq$) Let $x \in \reg U$, so $\downset \upset x \subseteq U$. By \cref{lem: pri-facts-intersection}, $\downset \upset x$ is an intersection of clopen downsets, and so by compactness there is a clopen downset $D$ such that $\upset x \subseteq D \subseteq U$. Using \cref{lem: pri-facts-intersection} again, $\upset x$ is an intersection of clopen upsets, so by compactness there is $V \in \clopup(X)$ such that $x \in V \subseteq D \subseteq U$. Thus,  $\downset V \subseteq U$.
\end{proof}

\begin{remark} \label{rem: containment of kernels}
    It is well known (see, e.g., \cite[p.~91]{PP21}) that regular frames 
    are both subfit and Hausdorff. In the language of \L-spaces, this can be seen as follows. If $X$ is an \L-space, then 
    for each $U \in \clopup(X)$, we have the following chain of inclusions: 
        \[\reg U \subseteq \hausfit U \subseteq \sfit U.\]
    The containment $\hausfit U \subseteq \sfit U$ is immediate. For $\reg U \subseteq \hausfit U$, suppose $x \in \reg U$. Then $\downset x \subseteq \downset \upset x \subseteq U$. Let $U \neq X$. Then $U^c \neq \varnothing$, so there is $y \in \max(U^c)$ (see \cref{lem: pri-facts-max}). If $\upset y \cap \upset x \neq \varnothing$, then $y \in \downset \upset x \subseteq U$, contradicting that $y \in U^c$. Thus, $\upset y \cap \upset x = \varnothing$, and hence $x\in\hausfit U$. 
\end{remark}

Complete regularity is an important strengthening of regularity and serves as a bridge to normality. In the pointfree setting it is obtained by refining the rather-below relation. 
\begin{definition}[{see, e.g., \cite[p.~91  ]{PP12}}]
    Let $L$ be a frame. 
    \begin{enumerate}
        \item For $a,b \in L$, a \emph{scale} from $a$ to $b$ is a family $\{a_r \mid r \in \mathbb{Q} \cap [0,1]\} \subseteq L$ such that 
        \[
            a = a_0 \leq a_r \prec a_s \leq a_1 = b 
        \]
        for all $r < s$. We say that $a$ is \emph{completely below} $b$, written $a \pprec b$, provided such a scale exists.\footnote{This relation 
        is also called the {\em really inside} relation (see, e.g., \cite[p.~126]{Joh82}).}
        \item $L$ is \emph{completely regular} if for all $a, b \in L$,
        \[
            a \nleq b \implies \exists c \in L : c \pprec a \text{ and } c \nleq b.
        \]
        Equivalently, $L$ is completely regular provided  
        $a = \bigvee\{c \in L \mid c \pprec a\}$ for all $a \in L$.
    \end{enumerate}
\end{definition}

\begin{remark}
        Assuming the Axiom of Countable Dependent Choice, the relation $\pprec$ is the largest relation $R$ contained in $\prec$ that {\em interpolates}, meaning that 
        from $aRb$ it follows that there is $c$ with $aRc$ and $cRb$ (see, e.g., \cite[p.~108]{PP21}). This will be used in \cref{prop: kernels coincide}.
        \label{rem: interpolates}
\end{remark}

There is a direct translation of the completely below relation in the language of Priestley spaces.

\begin{lemma} \label{lem: pprec}
    Let $X$ be an \L-space. For $U, V \in \clopup(X)$, $U \pprec V$ if and only if there is a family $\{U_r \mid r \in \mathbb Q \cap [0,1]\} \subseteq \clopup(X)$ such that 
    \[
        U = U_0 \subseteq U_r \subseteq \downset U_r \subseteq U_s \subseteq U_1 = V
    \]
    for $r < s$.
\end{lemma}
\begin{proof}
    Apply \cref{lem: prec}.
\end{proof}

\begin{definition}
    Let $X$ be an \L-space and $U \in \clopup(X)$. 
    \begin{enumerate}
        \item A point $x$ is \emph{completely inside} $U$ if $x \in U_0$ for some scale $\{U_r\}$ of $U$. 
        \item The \emph{completely regular kernel} of $U$ is
        \[
            \creg(U) = \{x \in U \mid x \text{ is completely inside }U\}.
        \]
    \end{enumerate}
\end{definition}

\begin{theorem} \label{thm: char cregular}
    Let $L$ be a frame and $X$ its \L-space. Then $L$ is completely regular if and only if $\cl (\creg U) = U$ for each $U \in \clopup(X)$.
\end{theorem}

\begin{proof}
    This follows from  \cref{prop: exact joins-1,lem: pprec}.
\end{proof}

\color{black}

\subsection{Normality}
We conclude this section by describing the \L-spaces of normal frames, and show that under this assumption all previously introduced kernels coincide.

\begin{definition}[{see, e.g., \cite[p.~91]{PP12}}]
    A frame $L$ is \emph{normal} if for all $a,b \in L$, 
    $$a \vee b = 1 \Longrightarrow \exists u,v \in L : u \wedge v = 0, \ a \vee v = 1, \mbox{ and } b \vee u = 1.$$
\end{definition}
The next theorem goes back to \cite[p.~68]{Joh82}, where the equivalence of (1), (3), and (4) is established.
\begin{theorem}
    Let $L$ be a frame and $X$ its \L-space. The following are equivalent.
    \begin{enumerate}[cref=theorem]
        \item $L$ is normal.
        \item For all disjoint clopen downsets $D,E \subseteq X$ there exist disjoint clopen upsets $U,V \subseteq X$ such that $D \subseteq U$ and $E \subseteq V$.
        \item For all $x,y \in X$, if $\upset x \cap \upset y \neq \varnothing$ then $\downset x \cap \downset y \neq \varnothing$. \label{thm: church rosser}
        \item $\min (\downset x)$ is a singleton for each $x \in X$; i.e., each prime filter of $L$ contains a unique minimal prime filter.
        \item $\upset D$ is a downset for each downset $D$. 
        \label{thm: up of clopen down}
    \end{enumerate}
\end{theorem}
\begin{proof}
    (1)$\Leftrightarrow$(2) This follows easily from the observation that $A,B$ are clopen upsets satisfying $A \cup B = X$ if and only if $A^c,B^c$ are disjoint clopen downsets.

    (2)$\Rightarrow$(3) Suppose $\downset x \cap \downset y = \varnothing$. Then by \cref{lem: pri-facts-intersection} and compactness, there exist disjoint clopen downsets $D,E$ such that $x \in D$ and $y \in E$. By (1), there are disjoint clopen upsets $U,V$ with $x \in U$ and $y \in V$, so $\upset x \cap \upset y = \varnothing$.

    (3)$\Rightarrow$(4) Let $x \in X$. By \cref{lem: pri-facts}(\hyperref[lem: pri-facts-upclosed]{2},\hyperref[lem: pri-facts-max]{5}), $\min(\downset x)$ is nonempty. Suppose $y,z \in \min(\downset x)$. Then $x \in \upset y \cap \upset z$, so $\upset y \cap \upset z \neq \varnothing$, and hence  (3) yields that $\downset y \cap \downset z \neq \varnothing$. Thus, $y = z$.

    (4)$\Rightarrow$(5) Suppose $x \in \downset \upset D$. Then $x \in \downset y$ for some $y \in \upset D$. Therefore, there is $z \in \min D$ such that $z \leq y$. Since $D$ is a downset, (4) implies that $\min (\downset y) = \{z\}$. Thus, $z \leq x$, so $x \in \upset D$, and hence $\upset D$ is a downset.
    
    (5)$\Rightarrow$(2) Suppose $D$ and $E$ are disjoint clopen downsets. Then $\upset D$ and $\upset E$ are downsets by (5), so $\upset D \cap \upset E$ is a downset. 
    Therefore, $\min (\upset D \cap \upset E) \subseteq \min X$.
    If $x \in \upset D \cap \upset E \cap \min X$, then there exist $y \in D$ and $z \in E$ such that $y \leq x$ and $z \leq x$. But then $y = x = z$ since $x \in \min X$, so $x \in D \cap E$. Thus, $\min(\upset D \cap \upset E) \subseteq D \cap E$. 
    If $\upset D \cap \upset E \neq \varnothing$, then $\min(\upset D \cap \upset E) \neq \varnothing$ by \cref{lem: pri-facts-max}. But then $D$ and $E$ cannot be disjoint, so $\upset D \cap \upset E = \varnothing$. By \cref{lem: pri-facts-intersection} and compactness, there are disjoint clopen upsets $U,V$ such that $D \subseteq \upset D \subseteq U$ and $E \subseteq \upset E \subseteq V$, completing the proof.
 \end{proof}

To show that all kernels coincide in \L-spaces of normal frames, we require the following lemma, which proves via Priestley duality the well-known fact (see, e.g., \cite[p.~138]{PP21}) that the rather-below relation $\prec$ interpolates in normal frames.

\begin{lemma} \label{lem: interpolate}
    Let $L$ be a normal frame and $X$ its \L-space.
    \begin{enumerate}[cref=lemma]
        \item For $U, V \in \clopup(X)$, if $\downset U \subseteq V$ then there is $W \in \clopup(X)$ such that 
        \[
            \downset U \subseteq W \subseteq \downset W \subseteq V.
        \]
        \item $\prec$ interpolates. \label{lem: interpolates}
    \end{enumerate}
\end{lemma}

\begin{proof}
    (1) Since $\downset U \subseteq V$ and $V$ is an upset,
    $\upset \downset U \subseteq V$. Because $L$ is normal, \cref{thm: up of clopen down} implies that $\upset \downset U$ is a downset. Therefore, by \cref{lem: pri-facts}(\hyperref[lem: pri-facts-upclosed]{2},\hyperref[lem: pri-facts-intersection]{3})  and compactness,
    there is a clopen downset $D$ such that $\upset \downset U \subseteq D \subseteq V$. Since $\upset \downset U$ is a closed upset contained in $D$, \cref{lem: pri-facts-intersection} and compactness yield 
    $W \in \clopup(X)$ such that $\upset \downset U \subseteq W \subseteq D$, so $\downset W \subseteq D$ because $D$ is a downset. Thus, $\downset U \subseteq W \subseteq \downset W \subseteq V$.

    (2) This follows from \cref{lem: prec} and (1).
\end{proof}

\begin{proposition} \label{prop: kernels coincide}
    Let $L$ be a normal frame and $X$ its \L-space. Then, for each $U \in \clopup(X)$, 
    \[\sfit U =\hausfit U = \reg U = \creg U.\]
\end{proposition}

\begin{proof}
    By \cref{lem: interpolates}, $\prec$ interpolates.
    Consequently, ${\prec} = {\pprec}$ (see \cref{rem: interpolates}). Therefore, $\creg U = \reg U$. Thus, \color{black}
    by \cref{rem: containment of kernels}, it suffices to show that $\sfit U \subseteq \reg U$. Suppose $x \in \sfit U$, so $\downset x \subseteq U$. Let $y \in \downset \upset x$. Then $\upset x \cap \upset y \neq \varnothing$, so by \cref{thm: church rosser}, $\downset x \cap \downset y \neq \varnothing$. Therefore, 
    there is $z \in \downset x \subseteq U$ such that $z \leq y$. Since $U$ is an upset, $y \in U$. Thus, $\upset \downset x \subseteq U$, and so $x \in \reg U$.
\end{proof}

\cref{prop: kernels coincide} 
yields an alternate proof of the well-known fact that subfit normal frames are completely regular (see, e.g., \cite[p.~138]{PP21}):

\begin{corollary}
    Subfit normal frames are completely regular.
\end{corollary}
\begin{proof}
    Suppose $L$ is normal and subfit. Let $X$ be the \L-space of $L$, and let ${U \in \clopup(X)}$. Then $\cl \sfit U = U$ by \cref{thm: subfit-5}. 
    Thus, $\cl \creg U = U$ by \cref{prop: kernels coincide}, and hence $L$ is completely regular by \cref{thm: char cregular}.
\end{proof}

\section{Compactness and local compactness} \label{sec: compact}

One of the central categories in topology is that of compact Hausdorff spaces. It illustrates how separation axioms often behave better under compactness. For instance, a compact space is Hausdorff if and only if it is regular.
However, this is no longer true in the pointfree setting since 
compact Hausdorff frames need not even be subfit (see, e.g., \cite[p.~46]{PP21}). Because of this, it is more natural to work with 
the category of compact regular frames. By Isbell duality \cite{Isb72} (see also \cite{BM80} and \cite[p.~90]{Joh82}), these exactly correspond 
to compact Hausdorff spaces. 
This result can be derived from Isbell's Spatiality Theorem (see, e.g., \cite[p.~216]{PP21}), which implies that every compact regular frame is spatial.

\subsection{Compactness and Isbell's spatiality theorem}

\begin{definition}[{see, e.g., \cite[p.~80]{Joh82}}]
    Let $L$ be a frame. We say that
    \begin{enumerate}
        \item $a \in L$ is \emph{compact} if for each $S \subseteq L$, from $a \leq \bigvee S$ it follows that $a \leq \bigvee T$ for some finite $T \subseteq S$;
        \item $L$ is \emph{compact} if its top element $1$ is compact.
    \end{enumerate}
\end{definition}

The Priestley spaces of compact frames have been characterized in \cite[Thm.~3.5]{PS88} and \cite[Lem.~3.1]{BGJ16}. These characterizations can be relativized to elements of~$L$:
\begin{proposition}[{see, e.g., \cite[Cor.~5.4(1)]{BM22}}] \label{prop: compact char}
    Let $L$ be a frame, $X$ its \L-space, and $a \in L$.  The following are equivalent.
    \begin{enumerate}[cref=proposition]
        \item $a$ is compact.
        \item $\sigma(a) \subseteq \cl U$ implies $\sigma(a) \subseteq U$ for all $U \in \opup(X)$. \label{prop: compact char-2}
        \item $\min \sigma(a) \subseteq \loc X$. \label{prop: compact char-3}
    \end{enumerate}
\end{proposition}
\begin{proof}
    (1)$\Rightarrow$(2) Let $U$ be an open upset with $\sigma(a)  \subseteq \cl U$. By \cref{lem: pri-facts-unions}, $U = \bigcup \sigma[S]$ for some $S \subseteq L$. Therefore, $\sigma(a) \subseteq \sigma(\bigvee S)$ by \cref{prop: exact joins-1}. Then $a \leq \bigvee S$, and since $a$ is compact, $a \leq \bigvee T$ for some finite $T \subseteq S$. Thus, $\sigma(a) \subseteq \bigcup \sigma[T] \subseteq U$.

    (2)$\Rightarrow$(3) Let $x \in \min \sigma(a)$. If $x \not \in \loc X$, then $U := X \setminus \downset x$ is not closed. Therefore, ${\cl(U) \cap \downset x \neq \varnothing}$, which means that $x \in \cl U$ since $\cl U$ is an upset (because $X$ is an \L-space). Thus, $\sigma(a) \subseteq \cl U$ since $\sigma(a)\setminus\{x\} \subseteq U$. But $\sigma(a) \nsubseteq U$ since $x\notin U$, contradicting (2).

    (3)$\Rightarrow$(1) Let $a \leq \bigvee S$. Then $\sigma(a) \subseteq \cl \bigcup \sigma[S]$ by \cref{prop: exact joins-1}. Suppose $x \in \min \sigma(a)$. Then $x \in \loc X$ by (3). Therefore,  $\downset x$ is open, so $\downset x \cap \bigcup \sigma[S] \neq \varnothing$, which means that $x \in \bigcup \sigma[S]$. Thus, $\min \sigma(a) \subseteq \bigcup \sigma[S]$, and hence $\sigma(a) \subseteq \bigcup \sigma[S]$ by \cref{lem: pri-facts-max}. Since $\sigma(a)$ is compact, there exists a finite $T \subseteq S$ such that $\sigma(a) \subseteq \bigcup \sigma[T]$. Consequently, $a \leq \bigvee T$, and hence $a$ is compact.
\end{proof}

As an immediate corollary, we obtain a characterization of Priestley spaces of compact frames. 

\begin{corollary}\label{cor: compact}
    Let $L$ be a frame and $X$ its \L-space. The following are equivalent.
    \begin{enumerate}[cref=corollary]
        \item $L$ is compact.
        \item $X= \cl U$ implies $X= U$ for all $U \in\opup(X)$. \label{thm: compact is scott upset}
        \item $\min X \subseteq \loc X$. \label{thm: compact and min}
    \end{enumerate}
\end{corollary}

\begin{remark}
    \cref{thm: compact is scott upset} is proved in \cite[Thm.~3.5]{PS88}, and \cref{thm: compact and min} in \cite[Lem.~3.1]{BGJ16}.
\end{remark}

We are ready to give a simple proof of the following well-known result of Isbell \cite[2.1]{Isb72}:

\begin{theorem}[Isbell's Spatiality Theorem] \label{thm: subfit compact is spatial}
    Compact subfit frames are spatial.
\end{theorem}
\begin{proof}
    Let $X$ be the \L-space of $L$. Since $L$ is compact, $\min X \subseteq \loc X$ by \cref{thm: compact and min}; and since $L$ is subfit, $\min X$ is dense by \cref{thm: subfit-3}. Thus, $\loc X$ is dense, and hence $L$ is spatial by \cref{thm: spatial}.
\end{proof}

\begin{remark}
    The classic proof of Isbell's Spatiality Theorem uses Zorn's Lemma. The above proof also relies on AC since we use 
    \cref{lem: pri-facts-max} (see \cref{rem: AC}). It remains open whether Isbell's Spatiality Theorem is in fact equivalent to~AC.
\end{remark}

\subsection{Compact regular frames and Isbell duality}

We now use the descriptions of compactness and regularity in terms of \L-spaces to derive  Isbell duality between compact regular frames and compact Hausdorff spaces. 

\begin{definition}[{see, e.g., \cite[Defs.~6.12 and 7.6]{BM23}}]
    Let $X$ be an \L-space. We call $X$
    \begin{enumerate}
        \item \emph{\L-compact} or a \emph{compact \L-space} if $\min X \subseteq \loc X$;
        \item \emph{\L-regular} or a \emph{regular \L-space} if  $\cl \reg U = U$ for each $U \in \clopup(X)$.
    \end{enumerate}
\end{definition}

By \cref{thm: compact and min}, compact \L-spaces are precisely the \L-spaces of compact frames, and by \cref{thm: char regular} regular \L-spaces are precisely the \L-spaces of regular frames. 
Let $\KRFrm$ be the full subcategory of $\Frm$ consisting of compact regular frames, and let \KRLPries be the full subcategory of \LPries consisting of compact regular \L-spaces. As an immediate consequence of the above observation,
Wigner--Pultr--Sichler duality restricts to the following duality (which follows from  
\cite{PS88} and is stated in the form below in {\cite[Cor.~7.8]{BM23}}):

\begin{theorem} \label{thm: KRFrm KRLPries}
    \KRFrm and \KRLPries are dually equivalent.
\end{theorem}

To establish Isbell duality, we need to restrict the equivalence of \cref{thm: adjunction in priestley}. We do this by showing that the localic parts of compact regular \L-spaces are precisely the compact Hausdorff spaces. We begin by connecting \L-compactness with compactness of the localic part.

\begin{proposition}[{\cite[Lem.~6.15]{BM23}}]\label{prop: loc X compact}
    Let $X$ be an \L-space. If $X$ is \L-compact, then $\loc X$ is compact. The converse holds when $X$ is \L-spatial.
\end{proposition}
\begin{proof}
    ($\Rightarrow$) Suppose $\loc X \subseteq \bigcup \{U \cap \loc X\mid U \in \mathcal U\}$ for $\mathcal U \subseteq \clopup(X)$ (recall \cref{rem: loc topology}). By \cref{thm: compact and min}, $X = \upset \min X \subseteq \upset \loc X \subseteq \bigcup \mathcal U$ . Since $X$ is compact, there exists a finite $\mathcal V \subseteq \mathcal U$ such that $X = \bigcup \mathcal V$. This implies that $\loc X \subseteq \bigcup\{U \cap \loc X \mid U \in \mathcal V\}$, and so $\loc X$ is compact.

    ($\Leftarrow$) Assume that $\loc X$ is dense (recall \cref{thm: spatial}). Suppose $\loc X$ is compact and $x \in \min X$. If $x \not \in \loc X$, then $$\loc X \subseteq \bigcup\{U \cap \loc X \mid x \notin U \in \clopup(X)\}.$$ Since $\loc X$ is compact, there are $U_1, \dots, U_n \in \clopup(X)$ with $\loc X \subseteq U_1 \cup \dots \cup U_n$. Because $\loc X$ is dense, $x \in \cl \loc X \subseteq U_1 \cup \dots \cup U_n$, a contradiction. Thus, $x \in \loc X$.
\end{proof}

To prove the analogous statement about regularity, we require the following: 

\begin{lemma}[{\cite[Lem.~7.4(1)]{BM23}}] \label{lem: loc inside reg}
    Let $X$ be a regular \L-space and $U \in \clopup(X)$. Then $U \cap \loc X \subseteq \reg U$. 
\end{lemma}
\begin{proof}
    Suppose $x \in U \cap \loc X$. Then $\downset x$ is open. Since $X$ is \L-regular, $\reg U$ is dense in $U$. 
    Therefore, $\downset x \cap \reg U \neq \varnothing$, and so  
    there is $y \in \reg U$ such that $x \in \upset y$. Thus, $\downset \upset x \subseteq \downset \upset y \subseteq U$, and hence $x \in \reg U$.
\end{proof}

\begin{proposition}[{\cite[Thm.~7.11]{BM23}}] \label{prop: reg loc}
    Let $X$ be an \L-space. If $X$ is \L-regular, then $\loc X$ is regular. The converse holds when $X$ is \L-spatial.
\end{proposition}

\begin{proof}
    Suppose $X$ is \L-regular. Let $x \in \loc X$ and $U \in \clopup(X)$ such that $x \notin U^c \cap \loc X$. Then $x \in U$, so $\downset \upset x \subseteq U$ by \cref{lem: loc inside reg}. Therefore, by applying \cref{lem: pri-facts-intersection} and compactness twice (as in the proof of \cref{lem: interpolate}),
    there exist $V,W \in\clopup(X)$ such that $x \in V \subseteq W^c \subseteq U$. Hence, $x \in V$, $U^c \subseteq W$, and $V \cap W = \varnothing$. Thus, $V \cap \loc X$ and $W \cap \loc X$ are the desired opens of $\loc X$ separating $x$ and $U^c \cap \loc X$.

    Assume that $\loc X$ is dense. Let $\loc X$ be regular, $U \in \clopup(X)$, and ${x \in U \cap \loc X}$. Then $x \notin U^c \cap \loc X$, so by regularity of $\loc X$, there exist $V,W \in \clopup(X)$ such that ${x \in V \cap \loc X}$, $U^c \cap \loc X \subseteq W \cap \loc X$, and $V \cap W \cap \loc X = \varnothing$. Since $\loc X$ is dense, $U^c \subseteq W$ and $V \cap W = \varnothing$. Therefore, $x \in V \subseteq W^c \subseteq U$, and hence $\downset \upset x \subseteq U$. Consequently, $U \cap\loc X \subseteq \reg U$, and using again that $\loc X$ is dense, we conclude that $\cl \reg U = U$. Thus, $X$ is \L-regular.
\end{proof}

Let $\KHaus$ be the full subcategory of $\Sob$ consisting of compact Hausdorff spaces.
Since regular frames are subfit (see \cref{rem: containment of kernels}), Isbell’s Spatiality Theorem implies that compact regular frames are spatial (see \cref{thm: subfit compact is spatial}). Similarly, compact regular \L-spaces are \L-spatial. Hence, we may restrict the equivalences of \cref{thm: spatial duality,thm: adjunction in priestley} to the compact regular setting (see, e.g., \cite[Cor.~7.12]{BM23}):

\begin{theorem}
        $\KRLPries$ is equivalent to $\KHaus$.
        \label{thm: KRLPries KHaus}
\end{theorem}

Putting \cref{thm: KRFrm KRLPries,thm: KRLPries KHaus} together yields 
the following classic result \cite{Isb72} (see also \cite{BM80} and \cite[p.~90]{Joh82}):

\begin{corollary}[Isbell duality]
    $\KRFrm$ is dually equivalent to $\KHaus$.
\end{corollary}

\subsection{Local compactness, continuity, and Hofmann--Lawson duality}
Local compactness is a convenient 
generalization of compactness that retains many of its nice properties. In particular, every compact Hausdorff space is locally compact, so there is a natural generalization of Isbell duality to the locally compact setting \cite[Rem.~4]{Ban80} (see also \cite[Cor.~6.10]{BR25}). The frame analogue of a locally compact space is a continuous frame, which is described in terms of the way-below relation.

\begin{definition}[{see, e.g., \cite[pp.~134--135]{PP12}}]
    Let $L$ be a frame.
    \begin{enumerate}
        \item For $a, b \in L$, we say that $a$ is \emph{way below} $b$ and write $a \ll b$ if for each $S \subseteq L$,
        \[
        b \leq \bigvee S \ \Longrightarrow \ \mbox{there is a finite } T \subseteq S \mbox{ such that }a \leq \bigvee T.
        \]
        \item $L$ is \emph{continuous} if for all $a,b \in L$,
        \[
            a \nleq b \implies \exists c \in L : c \ll a \text{ and } c \nleq b.
        \]
        Equivalently, $L$ is continuous provided
        $a = \bigvee \{c \in L \mid c \ll a\}$ for all $a \in L$. 
    \end{enumerate}
\end{definition}

\begin{lemma}[{\cite[Prop.~3.6]{PS88}}] \label{lem: way-below}
    Let $X$ be an \L-space and $U,V \in \clopup(X)$. Then $U \ll V$ if and only if for every $W \in \opup(X)$, 
    \[
        V \subseteq \cl W \implies U \subseteq W.
    \]
\end{lemma}
\begin{proof}
    ($\Rightarrow$) Let $U \ll V$ and let $W$ be an open upset such that $V \subseteq \cl W$. By \cref{lem: pri-facts-unions},  
    $V \subseteq \cl \bigcup \mathcal S$ for some $\mathcal S \subseteq \clopup(X)$. Hence, by \cref{prop: exact joins-1}, $V \leq \bigvee \mathcal W$. Therefore, since $U \ll V$, there is a finite $\mathcal T \subseteq \mathcal S$ such that $$U \leq \bigvee \mathcal T = \bigcup \mathcal T \subseteq W.$$

    ($\Leftarrow$) Let $V \leq \bigvee \mathcal S$ for some $\mathcal S \subseteq \clopup(X)$. Then $V \subseteq \cl \bigcup \mathcal S$ by \cref{prop: exact joins-1}, so $U \subseteq \bigcup \mathcal S$ by  assumption. By compactness, there is a finite $\mathcal T \subseteq \mathcal S$ such that $U \subseteq \bigcup \mathcal T$. Thus, $U \leq \bigvee \mathcal T$, and hence $U \ll V$.
\end{proof}

 A characterization of Priestley spaces of continuous frames 
 using kernels 
 was obtained in \cite[Sec.~5]{BM23}. 

\begin{definition}[{see, e.g., \cite[Defs.~5.2 and 5.4]{BM23}}]
     Let $X$ be an \L-space. 
    \begin{enumerate}
        \item For $U \in \clopup(X)$ define the \emph{continuous kernel} of $U$ as 
    \[
        \con U = \bigcup\{V \in \clopup(X) \mid \forall W \in\opup(X) : U \subseteq \cl W \Rightarrow V \subseteq W\}.
    \]
        \item We call $X$ \emph{\L-continuous} or a \emph{continuous \L-space} if $\cl \con U = U$ for each $U \in \clopup(X)$.
    \end{enumerate}
\end{definition}

\begin{theorem}[{\cite[Thm.~5.5]{BM23}}]
\label{thm: char continuous}
    Let $L$ be a frame and $X$ its \L-space. Then $L$ is continuous if and only if
    $\cl (\con U) = U$ for each $U \in \clopup(X)$.
\end{theorem}
\begin{proof}
    This follows from \cref{lem: way-below} and \cref{prop: exact joins-1}.
\end{proof}

It is common to consider those frame homomorphisms between continuous frames that preserve the way-below relation (see, e.g., \cite[p.~421]{GH+03}). Such frame homomorphisms are called \emph{proper}. 
    The Priestley counterpart of a proper frame homomorphism is described in \cite{BM23}. 
    Call an \L-morphism $f : X \to Y$ between \L-spaces \emph{\L-proper} or a \emph{proper \L-morphism} if \[
            f^{-1}(\con U) \subseteq \con f^{-1}(U)
        \]
        for all $U \in \clopup(Y)$.

\begin{proposition}[{\cite[Lem.~5.11]{BM23}}] \label{prop: proper morph}
    Let $h : L_1 \to L_2$ be a frame homomorphism and $f : X_2 \to X_1$ its dual \L-morphism. Then $h$ is proper if and only if $f$ is \L-proper.
\end{proposition}
\begin{proof}
    It suffices to show that $f^{-1}\colon \clopup(X_1)\to \clopup(X_2)$ is proper if and only if $f$ is \L-proper. 
    First, suppose $f^{-1}$ is proper. Let $U\in\clopup(X_1)$ and $x\in f^{-1}(\con U)$. By \cref{lem: way-below}, $f(x)\in V\ll U$ for some $V\in\clopup(X_1)$. Since $f^{-1}$ is proper, $x\in f^{-1}(V)\ll f^{-1}(U)$, so applying the same lemma again yields $x\in \con f^{-1}(U)$. Thus, $f^{-1}(\con U) \subseteq \con f^{-1}(U)$, and hence $f$ is \L-proper. 
    
    Conversely, suppose $f$ is \L-proper and $U\ll V$ in $\clopup(X_2)$. Then $U\subseteq \con V$, and since $f$ is \L-proper, 
    $
    f^{-1}(U)\subseteq f^{-1}(\con V)\subseteq \con f^{-1}(V).
    $
    Because $\con f^{-1}(V)$ is open, compactness and \cref{lem: way-below} give that $f^{-1}(U)\ll f^{-1}(V)$. 
    Thus,  $f^{-1}$ is proper.
\end{proof}

Let \ConFrm be the category of continuous frames and proper frame homomorphisms, and let $\ConLPries$ be the category of continuous \L-spaces and proper \L-morphisms.
Since isomorphisms in \ConFrm are proper and isomorphisms in \ConLPries are \L-proper, \cref{thm: WPS duality,thm: char continuous,prop: proper morph}
yield: 

\begin{corollary}[{\cite[Thm.~5.13]{BM23}}]
$\ConFrm$ is dually equivalent to $\ConLPries$. \label{cor: ConFrm ConLPries}
\end{corollary}

We next connect $\ConLPries$ with the category of locally compact sober spaces. Recall (see, e.g., \cite[p.~44]{GH+03}) that a topological space $X$ is \emph{locally compact} if for each open $U$ and $x \in U$ there exist an open $V$ and compact $K$ such that $x \in V \subseteq K \subseteq U$. Also recall (see, e.g., \cite[p.~43]{GH+03}) that a \emph{saturated} set is an intersection of open sets. We show that compact saturated subsets of $\loc X$ correspond to certain closed upsets of $X$.

\begin{definition}[{see, e.g., \cite[Def.~3.1]{BM22}}]
    Let $X$ be an \L-space. We call a closed upset $F \subseteq X$ a \emph{Scott upset} if $\min F \subseteq \loc X$. 
\end{definition}

\begin{proposition}[{\cite[Thm.~5.7]{BM22}}] \label{prop: HM}
    Let $X$ be an \L-space.
    \begin{enumerate}[cref=proposition]
        \item For each Scott upset $F$, $F \cap \loc X$ is compact saturated. \label{prop: HM-1}
        \item For each compact saturated $K \subseteq \loc X$, $\upset K$ is a Scott upset. \label{prop: HM-2}
        \item These assignments are inverse to each other.
    \end{enumerate}
\end{proposition}
\begin{proof}
    (1) By \cref{lem: pri-facts-intersection}, $F = \bigcap \{U \in \clopup(X) \mid F \subseteq U\}$. Therefore, 
    \[{F \cap \loc X = \bigcap\{U \cap \loc X \mid F \subseteq U \in \clopup(X)\}},\]
    so $F \cap \loc X$ is saturated. To see that it is compact, let $F \cap \loc X \subseteq \bigcup\{U \cap \loc X \mid U \in \mathcal U\}$ for some $\mathcal U \subseteq \clopup(X)$. Since $F$ is a Scott upset, $\min F \subseteq \bigcup \mathcal U$, so $F = \upset \min F \subseteq \bigcup \mathcal U$ by \cref{lem: pri-facts-max}. By compactness, $F \subseteq \bigcup \mathcal V$ for some finite $\mathcal V \subseteq \mathcal U$. Thus, $$F \cap \loc X \subseteq \bigcup \{U \cap \loc X \mid U \in \mathcal V\},$$ and hence $F \cap \loc X$ is compact.

    (2) Clearly, $\upset K$ is an upset and $\min \upset K \subseteq \loc X$. It remains to show that $\upset K$ is closed. Let $x \notin \upset K$. Then $y \nleq x$ for each $y \in K$. By Priestley separation, for each $y \in K$ there is $U_y \in \clopup(X)$ such that $y \in U_y$ and $x\notin U_y$. Therefore, $K \subseteq \bigcup \{U_y \cap \loc X \mid y \in K\}$. Since $K$ is compact, there is $U \in \clopup(X)$ such that $K \subseteq U$ and $x \notin U$. Thus, $\upset K \subseteq U$, and hence $\upset K$ is closed.

    (3) If $F$ is a Scott upset, $$F = \upset \min F \subseteq \upset (F \cap \loc X) \subseteq F.$$ 
    If $K \subseteq \loc X$ is compact saturated,  
    \begin{align*}
        K 
        &= \bigcap \{U \cap \loc X \mid K \subseteq U \in \clopup(X)\}&&\text{(since $K$ is saturated)} \\
        &= \bigcap\{U \in \clopup(X)\mid K \subseteq U \} \cap \loc X \\
        &= \upset K \cap \loc X && \text{(since $\upset K$ is closed).}
    \end{align*}
    Thus, the assignments in (1) and (2) are inverses of each other.
\end{proof}

\begin{remark} 
    \cref{prop: HM} is closely related to the Hofmann--Mislove Theorem (see, e.g., \cite[p.~144]{GH+03}), which states that compact saturated subsets of a sober space correspond to Scott-open filters of its frame of opens. A proof of this result using Priestley duality is given in \cite{BM22}.
\end{remark}

Let $U$ and $V$ be clopen upsets of an \L-space $X$. In general, the relation $U \ll V$ does not imply the existence of a Scott upset between them, not even when $X$ is \L-spatial. However, for continuous \L-spaces the situation improves: 

\begin{lemma}[{\cite[Sec.~5]{PS00}}] \label{lem: way below and Scott}
    Let $X$ be an \L-space and $U,V \in \clopup(X)$. If there is a Scott upset $F \subseteq X$ such that $U \subseteq F \subseteq V$ then $U \ll V$. The converse holds when $X$ is \L-continuous.
\end{lemma}
\begin{proof}
    See \cite[Prop.~8.13]{Mel25}.
\end{proof}

\begin{theorem}[{\cite[Prop.~4.6]{PS00}}] \label{thm: con is spatial}
    Continuous \L-spaces are \L-spatial.
\end{theorem}
\begin{proof}
    Let $X$ be a continuous \L-space. Suppose $U \cap V^c \neq \varnothing$ for $U,V \in \clopup(X)$. Since $X$ is \L-continuous, $\cl (\con U)  \cap V^c \neq \varnothing$. Therefore, there is $W \in \clopup(X)$ such that $W \ll U$ (recall \cref{lem: way-below}) and $W \cap V^c \neq \varnothing$. Thus, by \cref{lem: way below and Scott}, there is a Scott upset $F$ such that $W \subseteq F \subseteq U$. Consequently, $F \cap V^c \neq \varnothing$, and since $V^c$ is a downset, $\min(F) \cap V^c \neq \varnothing$. Since $F$ is a Scott upset, $\loc X \cap F \cap V^c \ne \varnothing$, and hence $\loc X \cap U \cap V^c \neq \varnothing$. This implies that $X$ is \L-spatial.
\end{proof}

The above theorem together with \cref{thm: spatial,thm: char continuous} yields an alternative proof of the well-known fact (see, e.g., \cite[p.~311]{Joh82}) that every continuous frame is spatial. The spatiality of continuous frames forms the foundation for Hofmann--Lawson duality between continuous frames and locally compact sober spaces. This duality was established via Priestley duality in \cite[Sec.~5]{BM23}. It requires connecting continuous \L-spaces to locally compact sober spaces.

\begin{theorem}[{\cite[Thm.~5.10]{BM23}}]
    Let $X$ be an \L-space. If $X$ is \L-continuous, then $\loc X$ is locally compact. The converse holds provided $X$ is \L-spatial. \label{thm: locally compact loc}
\end{theorem}
\begin{proof}
    ($\Rightarrow$) Suppose $X$ is \L-continuous. Let $x \in U \cap \loc X$ for some $U \in \clopup(X)$. 
    Then $\downset x \cap \con U \neq \varnothing$ since $\downset x$ is open and $\con U$ is dense in $U$. Therefore, there is $V \in \clopup(X)$ such that $V \ll U$ and $V \cap \downset x \neq \varnothing$ (see \cref{lem: way-below}). Thus, $x\in V$, and so $x \in \con U$.
    By 
    \cref{lem: way-below,lem: way below and Scott}, there exist $W \in \clopup(X)$ 
    and a Scott upset $F$ such that $x \in W \subseteq F \subseteq U$. Intersecting with $\loc X$ yields the desired open and compact sets (the latter by \cref{prop: HM-1}).

    ($\Leftarrow$) Assume that $X$ is \L-spatial. Let $\loc X$ be locally compact. Since $\loc X$ is dense, it suffices to show that $\loc X \cap U \subseteq \con U$ for each $U \in \clopup(X)$. If $x \in \loc X \cap U$, then since $\loc X$ is locally compact, there exist $V \in \clopup(X)$ and a compact saturated set $K \subseteq \loc X$  such that $x \in V \cap \loc X\subseteq K \subseteq U \cap \loc X$. Because $\loc X$ is dense, $V \subseteq K \subseteq U$, and by \cref{prop: HM-2}, $\upset K$ is a Scott upset between $V$ and $U$. Thus, $V \ll U$ by \cref{lem: way below and Scott}, and hence $x \in \con U$ by \cref{lem: way-below}.
\end{proof}

We recall from \cite[p.~481]{GH+03} that a continuous map $f :X \to Y$ between locally compact sober spaces is \emph{proper} provided 
$f^{-1}(K)$ 
is compact for every compact saturated $K \subseteq Y$.

\begin{theorem}[{\cite[Thms.~5.15 and 5.16]{BM23}}] \label{thm: proper cont}
    For an \L-morphism $f : X \to Y$ between continuous \L-spaces, the following are equivalent.
    \begin{enumerate}
        \item $f$ is \L-proper.
        \item $f^{-1}(F)$ is a Scott upset for each Scott upset $F \subseteq Y$.
        \item The restriction $f : \loc X \to \loc Y$ is a proper continuous map.
    \end{enumerate}
\end{theorem}

Let $\LCSob$ be the category of locally compact sober spaces and proper continuous maps. As an immediate consequence of \cref{thm: adjunction in priestley,thm: locally compact loc,prop: proper morph,thm: proper cont} we obtain:

\begin{corollary}[{\cite[Thm.~5.18]{BM23}}] \label{cor: ConLPries LCSob}
    $\ConLPries$ is equivalent to $\LCSob$.
\end{corollary}

Putting \cref{cor: ConFrm ConLPries,cor: ConLPries LCSob} together yields Hofmann--Lawson duality between continuous frames and locally compact sober spaces \cite{HL77}:

\begin{corollary} [Hofmann--Lawson]
    $\ConFrm$ is dually equivalent to $\LCSob$.
\end{corollary}

\begin{remark}  \label{rem: stably}
    An important subcategory of continuous frames is that of \emph{stably continuous frames} (see, e.g., \cite[p.~488]{GH+03}); that is, continuous frames $L$ such that $a \ll b$ and $a \ll c$ imply $a \ll b \wedge c$. Their Priestley spaces are characterized in \cite[Thm.~6.4]{BM23} as those continuous \L-spaces $X$ in which 
    $\con(U \cap V) = \con U \cap \con V$ for all $U,V \in \clopup(X)$. The latter condition is equivalent to the intersection of two Scott upsets being a Scott upset. In \cite[Sec.~6]{BM23} it is shown that this characterization yields well-known dualities for stably continuous frames (see, e.g., \cite[p.~488]{GH+03}).
\end{remark}

We end this section by generalizing Isbell duality to locally compact Hausdorff spaces.

\begin{theorem} \label{thm: char conreg}
    Let $L$ be a spatial frame and $X$ its \L-space. The following are equivalent.
    \begin{enumerate}
        \item $L$ is continuous and regular.
        \item $X$ is \L-continuous and \L-regular.
        \item $\loc X$ is locally compact and Hausdorff.
    \end{enumerate}
\end{theorem}
\begin{proof}
    For (1)$\Leftrightarrow$(2) apply \cref{thm: char continuous,thm: char regular}, and for     (2)$\Leftrightarrow$(3) observe that a locally compact space is Hausdorff if and only if it is regular and apply \cref{prop: reg loc,thm: locally compact loc}.
\end{proof}

\begin{definition}
    \leavevmode
    \begin{enumerate}
        \item Let $\ConRegFrm$ be the full subcategory of \ConFrm consisting of regular frames.
        \item Let \ConRegLPries be the full subcategory of \ConLPries consisting of regular \L-spaces.
        \item Let \LCHaus be the full subcategory of \LCSob consisting of Hausdorff spaces.
    \end{enumerate}    
\end{definition}
\cref{thm: char conreg} allows us to restrict the equivalences of \cref{cor: ConFrm ConLPries,cor: ConLPries LCSob}.

\begin{corollary}
    \leavevmode
    \begin{enumerate}[cref=corollary]
        \item \ConRegFrm and \ConRegLPries are dually equivalent. \label{cor: ConRegFrm ConRegLPries}
        \item \ConRegLPries and \LCHaus are equivalent.
        \label{cor: ConRegLPries LCHaus}
        \item \ConRegFrm and \LCHaus are dually equivalent.
        \label{cor: ConRegFrm LCHaus}
    \end{enumerate}
\end{corollary}

\begin{remark} 
    \leavevmode
    \begin{enumerate}[cref=remark]
        \item As we pointed out at the beginning of this subsection, \cref{cor: ConRegFrm LCHaus} follows from \cite[Rem.~4]{Ban80} (see also \cite[Cor.~6.10]{BR25}).
        \item It is well known that every continuous map between compact Hausdorff spaces is proper, and that so is every frame homomorphism between compact regular frames. Likewise, every \L-morphism between compact regular \L-spaces is \L-proper (see, e.g., \cite[Thm.~7.18]{BM23}). Thus,
    \KHaus, \KRFrm, and \KRLPries are full subcategories of \LCHaus, \ConRegFrm, and \ConRegLPries, respectively.\label{rem: fullsub}
    \end{enumerate}
\end{remark}

We summarize the dualities established in this section in \cref{fig: 2}, using the same notation as in \cref{fig: Frm LPries Top}.

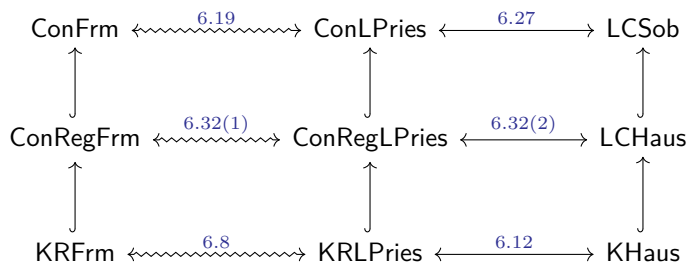
\begin{figure}[H]
\begin{center}
\begin{tikzcd}[column sep=huge, row sep=large]
\ConFrm 
\arrow[r, <->, dual, "\ref{cor: ConFrm ConLPries}"]
    & \ConLPries
    \ar[r, <->, "\ref{cor: ConLPries LCSob}"]
        & \LCSob
        \\
\ConRegFrm
\ar[u, hook]
\arrow[r, <->, dual, "\ref{cor: ConRegFrm ConRegLPries}"] 
    & \ConRegLPries
    \ar[u,hook]
    \ar[r, <->, "\ref{cor: ConRegLPries LCHaus}"]
        & \LCHaus
        \ar[u, hook] \\
\KRFrm
\ar[u,hook]
\arrow[r, <->, dual, "\ref{thm: KRFrm KRLPries}"] 
    & \KRLPries
    \ar[u, hook]
    \ar[r, <->, "\ref{thm: KRLPries KHaus}"]
        & \KHaus
        \ar[u, hook]
\end{tikzcd}
\end{center}
\caption{Equivalences between continuous frames, continuous \L-spaces, and locally compact sober spaces, together with their restrictions.}\label{fig: 2}
\end{figure}

\section{Algebraic, coherent, and Stone frames} \label{sec: algebraic}

A complete lattice is {\em algebraic} 
if the compact elements are {\em join dense} (that is, each element is a join of compact elements). 
It is a classical result of Nachbin \cite{Nac49} (see also \cite{BF48}) 
that algebraic lattices are precisely the ideal lattices of join-semilattices, and this correspondence restricts to algebraic frames and distributive join-semilattices (see, e.g., \cite[pp.~165, 167]{Gra11}). Restricting further to bounded distributive lattices and boolean algebras yields equivalences with coherent frames and Stone frames, respectively (see, e.g., \cite{Joh82,Ban89}).
In this section we describe the \L-spaces of algebraic, coherent, and Stone frames.
We conclude by deriving Priestley and Stone dualities from this approach, thereby coming back to where we started.

\subsection{Algebraic frames}

For a frame $L$, let $\K(L)$ denote the set
of compact elements of $L$. 

\begin{definition}[{see, e.g., \cite[p.~142]{PP12}}]\leavevmode
    \begin{enumerate}
        \item A frame $L$ is 
        \emph{algebraic} if for all $a,b \in L$,
    \[
        a \nleq b \implies \exists c \in \K(L) : c \leq a \text { and } c\nleq b.
    \]
    Equivalently, $L$ is algebraic provided the compact elements are join dense.
    \item Let \AlgFrm be the full subcategory of \ConFrm consisting of algebraic frames.
    \end{enumerate}
\end{definition}

\begin{remark}
    Recall (see, e.g., \cite[p.~64]{Joh82}) that a frame homomorphism $h : L \to M$ is \emph{coherent} if $h(k) \in \K(M)$ for every $k \in \K(L)$. 
    It is known that a frame homomorphism between algebraic frames is coherent if and only if it is proper (see, e.g., \cite[p.~271]{GH+03}).
    Hence, the morphisms of \AlgFrm are precisely the coherent frame homomorphisms.
\end{remark}

For an \L-space $X$, we denote by $\clopsup(X)$ the set of clopen Scott upsets of $X$. 

\begin{definition}[{\cite[Def.~4.2]{BM25}}]
    Let $X$ be an \L-space.
    \begin{enumerate}
        \item For $U \in \clopup(X)$ define the \emph{algebraic kernel} of $U$ as 
        \[
            \alg U = \bigcup \{V \in \clopsup(X) \mid V \subseteq U\}. 
        \]
        \item Call $X$  \emph{\L-algebraic} or an \emph{algebraic \L-space} if $\cl \alg U = U$ for each ${U \in \clopup(X)}$.
    \end{enumerate}
\end{definition}

To show that algebraic \L-spaces are the Priestley duals of algebraic frames, we require the following: 

\begin{lemma}[{see, e.g., \cite[Lem.~4.12]{BM25}}] \label{lem: compact opens}
    Let $X$ be a spatial \L-space. If $K \subseteq \loc X$ is compact open, then $\cl K \in \clopsup(X)$. 
\end{lemma}

Following \cite[p.~2063]{Ern09}, we call a topological space \emph{compactly based}
if it has a basis of compact open sets.

\begin{theorem}[{\cite[Thm.~4.5]{BM25} and \cite[Thm.~12.11]{Mel25}}]\label{thm: char algebraic}
    Let $L$ be a frame and $X$ its \L-space. The following are equivalent.
    \begin{enumerate}
        \item $L$ is algebraic.
        \item $X$ is \L-algebraic.
    \end{enumerate}
    Each of these implies
    \begin{enumerate}[resume]
        \item $\loc X$ is compactly based. 
    \end{enumerate}
    If $L$ is spatial, the above three conditions are equivalent.
\end{theorem}
\begin{proof}
    (1)$\Leftrightarrow$(2) By
    \cref{prop: compact char}, $\K(\clopup(X)) = \clopsup(X)$. Thus, by \cref{prop: exact joins-1}, $\cl (\alg \sigma(a)) = \sigma(a)$ if and only if $a = \bigvee \{k \in \K(L) \mid k \leq a\}$. Hence, $L$ is algebraic if and only if $X$ is \L-algebraic.

    (2)$\Rightarrow$(3) Suppose $X$ is \L-algebraic. Then
    $U \cap \loc X = \cl (\alg U) \cap \loc X$ for each $U \in \clopup(X)$. 
    If $x \in U \cap \loc X$, then $\downset x$ is open and hence $\downset x \cap \alg U \neq \varnothing$. 
    Therefore, there is $V \in \clopsup(X)$ such that $V \subseteq U$ and $V \cap \downset x\neq \varnothing$. Thus, $x \in V$. 
    But $V \cap \loc(X)$ is compact open by \cref{prop: HM-1}. Hence, $\loc X$ is compactly based.   

    Now suppose that $L$ is spatial, so $X$ is \L-spatial (see \cref{thm: spatial}). It is sufficient to show that (3) implies (2). 

    (3)$\Rightarrow$(2) Let $U \in \clopup(X)$. Then $U \cap \loc X$ is open. Since $\loc X$ is compactly based, 
    there exist compact opens $\{K_i\} \subseteq \loc X$ such that $U \cap \loc X = \bigcup K_i$, so $U = \cl \bigcup K_i$ since $X$ is \L-spatial. By \cref{lem: compact opens}, $\cl K_i \in \clopsup(X)$ for each $i$. Thus,
    \[
        U = \cl \medcup K_i \subseteq \cl\medcup \cl K_i \subseteq \cl \alg U,
    \]
    and hence $X$ is \L-algebraic.
\end{proof}

Let \AlgLPries be the full subcategory of \ConLPries consisting of algebraic \L-spaces, and let $\KBSob$ be the full subcategory of $\LCSob$ consisting of compactly based spaces. Using \cref{thm: char algebraic}, we can restrict \cref{cor: ConFrm ConLPries,cor: ConLPries LCSob} to obtain:
\begin{corollary}[{\cite[Thm.~4.11 and Cor.~4.14]{BM25}}]\leavevmode\label{cor: alg dualities}
    \begin{enumerate}
        \item $\AlgFrm$ is dually equivalent to $\AlgLPries$.
        \item $\AlgLPries$ is equivalent to $\KBSob$.
    \end{enumerate}
\end{corollary}

An immediate consequence is the following result of Hofmann and Keimel (see also \cite[p.~423]{GH+03}).

\begin{corollary}[{\cite[Thm.~5.7]{HK72}}] \label{cor: algfrm kbsob}
    $\AlgFrm$ is dually equivalent to $\KBSob$.
\end{corollary}

\begin{remark} \label{rem: arithmetic}
    Recall \cite[p.~117]{GH+03} that an algebraic frame is \emph{arithmetic} if the meet of any two compact elements is compact. Arithmetic frames form a full 
    subcategory of stably continuous frames (see \cref{rem: stably}). The Priestley spaces of arithmetic frames are  the algebraic \L-spaces $X$ for which $\alg$ commutes with binary intersections 
    \cite[Thm.~13.5]{Mel25} (see also \cite[Sec.~5]{BM25}). This characterization yields the known duality for arithmetic frames \cite[Thm.~5.7]{HK72} (see also \cite[p.~423]{GH+03}).
\end{remark}

\subsection{Coherent frames} \label{subsec: coherent}

We now turn to coherent frames and their \L-spaces.  This essentially requires characterizing stability (see \cref{rem: stably}) in the setting of algebraic frames (see \cref{rem: arithmetic}).

\begin{definition}[{see, e.g., \cite[p.~65]{Joh82}}]
    \leavevmode
\begin{enumerate}
    \item
    A frame $L$ is \emph{coherent} if it is a compact arithmetic frame. 
     Equivalently, $L$ is coherent provided $K(L)$ forms a bounded sublattice of $L$ which is join dense. 
     \item Let \CohFrm be the full subcategory of \AlgFrm consisting of coherent frames.
\end{enumerate}
\end{definition}

\begin{remark} \label{rem: K(L)}
    For each frame $L$, we have that $K(L)$ is closed under finite joins.
    Thus, $\K(L)$ always forms a join-semilattice, which is always distributive (see, e.g., \cite[Lem.~184]{Gra11}). In addition, 
    \begin{itemize}
        \item $L$ is algebraic if $\K(L)$ is join dense;
        \item $L$ is arithmetic if moreover
        $\K(L)$ is a lattice;
        \item $L$ is coherent if furthermore
        $\K(L)$ is a bounded lattice.
    \end{itemize}
\end{remark}

\begin{definition}[{\cite[Def.~5.6]{BM25}}]
\leavevmode
    \begin{enumerate}
        \item An \L-space $X$ is \emph{\L-coherent} or a \emph{coherent \L-space} if it is \L-algebraic, \L-compact, and
        \[
            \alg U \cap \alg V = \alg(U \cap V) \tag{B} \label{eq:B}
        \]
        for all $U, V \in \clopup(X)$.
        \item Let $\CohLPries$ be the full subcategory of $\AlgLPries$ consisting of coherent \L-spaces.
    \end{enumerate}
\end{definition}

\begin{lemma}[{\cite[Rem.~2.4]{BBM25} and \cite[Proof of Lem.~5.2]{BM25}}]
    Let $X$ be an \L-space. 
    \begin{enumerate}[cref=lemma]
        \item If $U \in \clopup(X)$, then $U \in \clopsup(X)$ if and only if  $\alg U = U$. \label{lem: alg fixpoint}
        \item $\clopsup(X)$ is closed under binary intersections if and only if \eqref{eq:B} is satisfied for all ${U,V \in \clopup(X)}$. \label{lem: alg stability}
    \end{enumerate}
\end{lemma}

We recall (see, e.g., \cite[p.~4]{DST19}) that a \emph{spectral space} is a compactly based sober space in which finite intersections of compact opens are compact. 

\begin{theorem}[{\cite[Thm.~5.8]{BM25}}] \label{thm: char coherent}
    Let $L$ be a spatial frame and $X$ its \L-space. The following are equivalent.
    \begin{enumerate}
        \item $L$ is coherent.
        \item $X$ is \L-coherent.
        \item $\loc X$ is a spectral space.
    \end{enumerate}
\end{theorem}

\begin{proof}
    Recall that $\loc X$ is always sober (see \cref{rem: sober}). Therefore, by putting \cref{prop: compact char,prop: loc X compact,thm: char algebraic} together, we may assume that $L$ is compact and algebraic, that $X$ is \L-compact and \L-algebraic, and that $\loc X$ is compact, compactly based, and sober.

    (1)$\Leftrightarrow$(2) 
    Since $K(L) \cong \K(\clopup(X)) = \clopsup(X)$ (see \cref{prop: compact char}), $L$ is coherent if and only if $\clopsup(X)$ is closed under binary intersections. 
    The result then follows from \cref{lem: alg stability}.

    (2)$\Rightarrow$(3) Let $U,V \in \clopup(X)$ be such that $U \cap \loc X$ and $V \cap \loc X$ are compact open. Since $X$ is \L-spatial (see \cref{thm: spatial}), $\cl (U\cap \loc X) = U$ and $\cl (V \cap \loc X)= V$. Thus, by \cref{lem: compact opens}, $U,V \in \clopsup(X)$. Hence, $U \cap V \in \clopsup(X)$ by \cref{lem: alg stability}, so $U \cap V \cap \loc X$ is compact by \cref{prop: HM-1}.

    (3)$\Rightarrow$(2) Let $U,V \in \clopsup(X)$. Then $U \cap \loc X$ and $V \cap \loc X$ are compact open by \cref{prop: HM-1}. Hence, $U \cap V \cap \loc X$ is compact open since $\loc X$ is spectral. Because $X$ is \L-spatial, $U \cap V = \cl(U \cap V \cap \loc X) \in \clopsup(X)$ by \cref{lem: compact opens}.
\end{proof}

Let $\Spec$ be the full subcategory of $\KBSob$ consisting of spectral spaces. By \cref{thm: char coherent}, \cref{cor: alg dualities} restricts to the following result:
\begin{corollary}[{\cite[Cor.~5.9]{BM25}}] \label{cor: coh dualities}
\leavevmode
    \begin{enumerate}[cref=corollary]
        \item $\CohFrm$ is dually equivalent to $\CohLPries$. \label{cor: coh dualities-1}
        \item $\CohLPries$ is equivalent to $\Spec$.\label{cor: coh dualities-2}
    \end{enumerate}
\end{corollary}

As a consequence, we arrive at the following well-known result
\cite[Thm.~5.7]{HK72} (see also \cite{Ban80,Ban81} and \cite[pp.~65--66]{Joh82}). 

\begin{corollary} \label{cor: Spec}
    $\CohFrm$ is dually equivalent to $\Spec$. 
\end{corollary}

Composing the equivalence of \cref{cor: Spec} with the Cornish isomorphism 
between \Spec and \Pries (recall the introduction),
we obtain that
\CohFrm is dually equivalent to \Pries. We will return to this in \cref{sec: deriving Priestley}.

\subsection{Stone frames}

We now turn our attention to the \L-spaces of Stone frames. This requires describing zero-dimensionality and complemented elements in terms of Priestley spaces. 
We recall (see the paragraph before \cref{thm: center}) that a frame is zero-dimensional if each element is a join of complemented elements. 

\begin{definition}[{see, e.g., \cite[p.~258]{Ban89}}]
\leavevmode
    \begin{enumerate}
        \item A \emph{Stone frame} is a zero-dimensional compact frame.
        \item Let $\StoneFrm$ be the full subcategory of $\Frm$ consisting of Stone frames.
    \end{enumerate}
\end{definition}

Clearly $\StoneFrm$ is a full subcategory of $\KRFrm$. This implies that $\StoneFrm$ is a full subcategory of $\ConFrm$ (see \cref{rem: fullsub}), and hence a full subcategory of $\CohFrm$.

We next turn our attention to the Priestley spaces of zero-dimensional and Stone frames. Those were described in 
\cite[Sec.~6]{BGJ16} and 
\cite[Sec.~5]{BM25}. 

\begin{definition}[{\cite[Def.~5.10]{BM25}}]
    Let $X$ be an \L-space.
    \begin{enumerate}
        \item For $U \in \clopup(X)$, define the \emph{zero-dimensional kernel} of $U$ as
        \[
            \zer U = \bigcup \{V \in \clopup(X) \mid V \subseteq U \text{ and } V \text{ is a downset}\}.
        \]
        \item We call $X$ \emph{\L-zero-dimensional} or a \emph{zero-dimensional \L-space} if $\cl (\zer U) = U$ for every ${U \in \clopup(X)}$.
        \item We call $X$ \emph{\L-Stone} or a \emph{Stone \L-space} if it is \L-compact and \L-zero-dimensional.
        \item Let \StoneLPries be the full subcategory of \LPries consisting of Stone \L-spaces.
    \end{enumerate}
\end{definition}

\begin{lemma}[{\cite[Lem.~6.1]{BGJ16}}] \label{lem: complemented}
    Let $X$ be an \L-space and $U \in \clopup(X)$. Then $U$ is complemented if and only if $U$ is a downset.
\end{lemma}

The next result was established in \cite[Thm.~6.3]{BGJ16} and \cite[Cor.~5.17]{BM25}.

\begin{theorem}
\label{thm: char zero-dim}
    Let $L$ be a frame and $X$ its \L-space. The following are equivalent.
    \begin{enumerate}
        \item $L$ is zero-dimensional.
        \item $X$ is \L-zero-dimensional.
    \end{enumerate}
    Each of these implies
    \begin{enumerate}[resume]
        \item $\loc X$ is zero-dimensional.
    \end{enumerate}
    If $L$ is spatial, the above three conditions are equivalent.
\end{theorem}

\begin{proof}
    (1)$\Leftrightarrow$(2) This follows from \cref{prop: exact joins-1,lem: complemented}.

    (2)$\Rightarrow$(3) Suppose $U \in \clopup(X)$ is a downset. Then $U \cap \loc X$ is clopen in $\loc X$. Hence, if $X$ is \L-zero-dimensional, then $\loc X$ has a basis of clopen sets (see \cref{rem: loc topology}).

    Next suppose that $L$ is spatial, so $X$ is \L-spatial (see \cref{thm: spatial}). It is sufficient to show that (3) implies (2). 
    
    (3)$\Rightarrow$(2) Suppose $U \in \clopup(X)$. Since $X$ is \L-spatial, $\loc X$ is dense, so it is enough to show that ${U \cap \loc X \subseteq \zer U}$. Let $x \in U \cap \loc X$. Because $\loc X$ is zero-dimensional, there is a clopen subset $V$ of $\loc X$ such that $x \in V$ and $V \subseteq U \cap \loc X$. 
    Since $V$ and $\loc X \setminus V$ are open, $V = W \cap \loc X = D \cap \loc X$ for a clopen upset $W$ and a clopen downset $D$ (see \cref{rem: loc topology}). Because $X$ is \L-spatial, $W = \cl V = D$, so $W$ is a downset. Moreover, $x \in W = \cl V \subseteq U$. Thus, $x \in \zer U$, and hence 
    $X$ is \L-zero-dimensional.  
\end{proof}

The above theorem together with 
\cref{cor: compact,prop: loc X compact} gives: 
\begin{corollary}[{\cite[Cor.~5.18]{BM25}}] \label{cor: char Stone}
    Let $L$ be a spatial frame and $X$ its \L-space. 
    The following are equivalent.
    \begin{enumerate}
        \item $L$ is a Stone frame.
        \item $X$ is a Stone \L-space.
        \item $\loc X$ is a Stone space.
    \end{enumerate}
\end{corollary}

Letting $\Stone$ denote the full subcategory of $\Top$ consisting of Stone spaces and applying \cref{cor: char Stone}, the duality of \cref{thm: spatial duality} restricts to yield:

\begin{corollary}[{\cite[Thm.~5.14 and Cor.~5.19]{BM25}}] \label{cor: stone dualities}
\leavevmode
    \begin{enumerate}[cref=corollary]
        \item $\StoneFrm$ is dually equivalent to $\StoneLPries$. \label{cor: stone dualities-1}
        \item $\StoneLPries$ is equivalent to $\Stone$.\label{cor: stone dualities-2}
    \end{enumerate}
\end{corollary}

As an immediate consequence, we obtain the following well-known result 
(see, e.g., 
\cite{Sto36,Ban89}). 

\begin{corollary} 
    $\StoneFrm$ is dually equivalent to $\Stone$. 
\end{corollary}

\begin{remark}
We call a topological space \emph{locally Stone} if it is zero-dimensional, locally compact, and Hausdorff (see, e.g., \cite[Def.~4]{Rum10}). 
A \emph{locally Stone frame} is an algebraic zero-dimensional frame (see, e.g., \cite[Def.~7.6]{BCM24}). 
Define a \emph{locally Stone \L-space} 
as an algebraic zero-dimensional \L-space. Then, using \cref{thm: char zero-dim}, the equivalences of
\cref{cor: alg dualities} restrict to the locally Stone case, yielding a generalization of \cref{cor: stone dualities}. An immediate corollary is the folklore duality between locally Stone frames and locally Stone spaces (see, e.g., \cite[Thm.~3.11]{BK25}). 
\end{remark}

\subsection{Deriving Priestley duality} \label{sec: deriving Priestley}

Let $L$ be an algebraic frame. 
If $L$ is coherent, then $\K(L)$ is a bounded distributive lattice (recall \cref{rem: K(L)}). This correspondence is at the heart of 
the well-known equivalence between $\CohFrm$ and $\DLat$, 
which further restricts to the well-known equivalence between $\StoneFrm$ and the category $\BA$ of boolean algebras and boolean homomorphisms: 

\begin{theorem}
\leavevmode
\begin{enumerate}[cref=theorem]
    \item {\em \cite[p.~65]{Joh82}} \CohFrm and \DLat are equivalent.\label{thm: cohfrm dlat}
    \item {\em \cite[p.~258]{Ban89}} \StoneFrm and \BA are equivalent. \label{thm: Ban Stone}
\end{enumerate}
\end{theorem}

By \cref{cor: coh dualities-1}, $\DLat$ is dually equivalent to \ConLPries, and hence to \Spec by \cref{cor: coh dualities-2}.
Similarly, by \cref{cor: stone dualities-1}, $\BA$ is dually equivalent to $\StoneLPries$, and hence to \Stone by \cref{cor: stone dualities-2}.
Consequently, we recover Stone’s classic dualities for boolean algebras \cite{Sto36} and distributive lattices \cite{Sto38}:

\needspace{2em}
\begin{corollary}[Stone dualities]
    \leavevmode
    \begin{enumerate}
        \item \BA is dually equivalent to \Stone.
        \item \DLat is dually equivalent to \Spec.
    \end{enumerate}
\end{corollary}

Finally, by Cornish’s isomorphism between \Spec and \Pries (see the introduction),
we obtain that $\DLat$ is dually equivalent to \Pries. 
In other words, the results of this section recover Priestley duality (see \cref{fig:PD}).

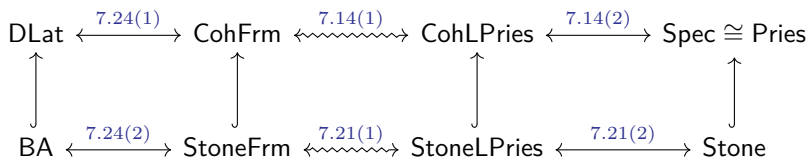
\begin{figure}[h]
    \centering
    \begin{tikzcd}[column sep=large, row sep=large]
        \DLat \ar[r, <->, "\ref{thm: cohfrm dlat}"] & \CohFrm \ar[r, <->, dual, "\ref{cor: coh dualities-1}"]& \CohLPries \ar[r, <->, "\ref{cor: coh dualities-2}"] & \Spec \cong \Pries \\
        \BA \ar[u, hook] \ar[r, <->, "\ref{thm: Ban Stone}"] & \StoneFrm \ar[u, hook] \ar[r, <->, dual, "\ref{cor: stone dualities-1}"] & \StoneLPries \ar[u, hook] \ar[r, <->,  "\ref{cor: stone dualities-2}"] & \Stone \ar[u, hook]
    \end{tikzcd}
     \caption{Recovering Priestley and Stone dualities from the dualities for coherent and Stone frames.}
    \label{fig:PD}
\end{figure}

We find it apt to end with the duality underlying this approach.

\begin{corollary}[Priestley duality]
    \DLat is dually equivalent to \Pries.
\end{corollary}

\bibliographystyle{alpha-init}
\bibliography{refs}

\end{document}

%% file: preamble.tex
\usepackage[a4paper, margin=1.2in]{geometry}
\usepackage{needspace}
\usepackage{xspace}

\usepackage[utf8]{inputenc}
\usepackage[T1]{fontenc}
\usepackage{lmodern}
\usepackage{amssymb}
\usepackage{mathrsfs}
\DeclareMathAlphabet{\dutchcal}{U}{dutchcal}{b}{n}

\usepackage[dvipsnames]{xcolor}

\usepackage[shortlabels]{enumitem}
\setlist[enumerate,1]{label={\upshape(\arabic*)},
ref=\thetheorem(\arabic*)}
\newcommand\crefenumtype{} 
\SetEnumitemKey{cref}{
  before={\def\crefenumtype{#1}
  \let\oldlabel\label
  \renewcommand{\label}[2][]{
    \ifx\crefenumtype\empty
      \oldlabel{####2}
    \else
      \oldlabel[\crefenumtype]{####2}
    \fi}
  }
}

\usepackage{tikz-cd}
\usepackage{float}
\usetikzlibrary{decorations.pathmorphing}
\usetikzlibrary{decorations.markings}
\usetikzlibrary{decorations.pathreplacing}
\usetikzlibrary{arrows.meta}
\usetikzlibrary{calc,positioning}
\tikzcdset{
  dual/.style={squiggly}
}

\usepackage[numbers]{natbib}
\usepackage{hyperref}
\hypersetup{linktoc=all,colorlinks=true,allcolors=Blue,breaklinks=true}
\usepackage[capitalize,noabbrev]{cleveref}
\newcommand{\bibsortkey}[1]{}

\newcounter{dummyitem}
\newcommand\myitem[1][]{\item[#1]\refstepcounter{dummyitem}\def\@currentlabel{#1}}

\makeatletter
\patchcmd{\@setaddresses}{\indent}{\noindent}{}{}
\patchcmd{\@setaddresses}{\indent}{\noindent}{}{}
\patchcmd{\@setaddresses}{\indent}{\noindent}{}{}
\patchcmd{\@setaddresses}{\indent}{\noindent}{}{}
\makeatother

\newtheorem{lemma}{Lemma}[section]
\newtheorem{theorem}[lemma]{Theorem}
\newtheorem{proposition}[lemma]{Proposition}
\newtheorem{corollary}[lemma]{Corollary}
\theoremstyle{definition}
\newtheorem{definition}[lemma]{Definition}
\newtheorem{remark}[lemma]{Remark}

\let\Sec\S

\newcommand{\upset}{\mathord{\uparrow}}
\newcommand{\downset}{\mathord{\downarrow}}

\newcommand{\pprec}{\mathrel{{\prec}\mkern-8mu{\prec}}}
\newcommand{\medcup}{{\textstyle\bigcup}}
\newcommand{\pt}{\mathord{\mathsf{pt}}}
\renewcommand{\o}{\mathfrak o}
\renewcommand{\c}{\mathfrak c}
\renewcommand{\b}{\mathfrak b}
\newcommand{\B}{\mathord{\mathfrak B}}
\renewcommand{\S}{\mathord{\mathsf S}}
\newcommand{\N}{\mathord{\mathsf N}}

\newcommand{\cat}[1]{{\sf #1}\xspace}
\newcommand{\DLat}{\cat{DLat}}
\newcommand{\BA}{\cat{BA}}
\newcommand{\Pries}{\cat{Pries}}
\newcommand{\LPries}{\cat{LPries}}
\newcommand{\SLPries}{\cat{SLPries}}
\newcommand{\ConLPries}{\cat{ConLPries}}
\newcommand{\ConRegLPries}{\cat{ConRegLPries}}

\newcommand{\KRLPries}{\cat{KRLPries}}
\newcommand{\AlgLPries}{\cat{AlgLPries}}

\newcommand{\CohLPries}{\cat{CohLPries}}
\newcommand{\StoneLPries}{\cat{StoneLPries}}
\newcommand{\Frm}{\cat{Frm}}
\newcommand{\SFrm}{\cat{SFrm}}
\newcommand{\ConFrm}{\cat{ConFrm}}
\newcommand{\ConRegFrm}{\cat{ConRegFrm}}

\newcommand{\KRFrm}{\cat{KRFrm}}
\newcommand{\AlgFrm}{\cat{AlgFrm}}

\newcommand{\CohFrm}{\cat{CohFrm}}
\newcommand{\StoneFrm}{\cat{StoneFrm}}
\newcommand{\Top}{\cat{Top}}
\newcommand{\Sob}{\cat{Sob}}
\newcommand{\LCSob}{\cat{LCSob}}

\newcommand{\KHaus}{\cat{KHaus}}
\newcommand{\LCHaus}{\cat{LCHaus}}
\newcommand{\KBSob}{\cat{KBSob}}
\newcommand{\Spec}{\cat{Spec}}
\newcommand{\Stone}{\cat{Stone}}

\renewcommand{\L}{\ensuremath{\mathsf{L}}\xspace}
\DeclareMathOperator{\reg}{reg}

\DeclareMathOperator{\sfit}{sfit}
\DeclareMathOperator{\creg}{creg}
\DeclareMathOperator{\haus}{haus}
\DeclareMathOperator{\con}{con}
\DeclareMathOperator{\zer}{zer}
\DeclareMathOperator{\alg}{alg}
\newcommand{\hausfit}{\haus_{\sfit}}
\newcommand{\K}{\operatorname{K}}

\newcommand{\loc}{\operatorname{loc}}
\newcommand{\clopup}{\operatorname{ClopUp}}
\newcommand{\clopsup}{\operatorname{ClopSUp}}
\newcommand{\opup}{\operatorname{OpUp}}
\newcommand{\cl}{\operatorname{cl}}

\newcommand{\functor}[1]{\dutchcal{#1}}
\newcommand{\Loc}{\functor{L\kern-3pto\kern-1ptc}}
\newcommand{\Pri}{\functor{P\kern-1ptr\kern-1.5pti}}

\renewcommand{\>}{\textup{)}}